\documentclass[12pt]{amsart}
\usepackage{amscd, amsfonts, amssymb}
\usepackage{fullpage}
\usepackage{latexsym}
\usepackage{verbatim}
\usepackage{enumerate}
\usepackage[all,cmtip]{xy}

\newcommand{\ds}{\displaystyle}
\newcommand\ra{\rightarrow}
\newcommand{\iso}{\cong}
\numberwithin{equation}{section}

\newtheorem{thm}[equation]{Theorem}

\newtheorem{lem}[equation]{Lemma}
\newtheorem{cor}[equation]{Corollary}
\newtheorem{prop}[equation]{Proposition}
\newtheorem{qn}[equation]{Question}

\theoremstyle{definition}
\newtheorem{defn}[equation]{Definition}
\newtheorem{ex}[equation]{Example}

\theoremstyle{remark}
\newtheorem{rem}[equation]{Remark}
\theoremstyle{remark}

\newcommand{\ovl}{\overline}

\subjclass[2010]{20C20 (20G15)}
\keywords{Representations of algebraic groups; reductive algebraic groups; conjugacy classes; nonabelian $1$-cohomology}

\date{October 10, 2017}

\title[On a question of K\"ulshammer for homomorphisms of algebraic groups]
{On a question of K\"ulshammer for homomorphisms of algebraic groups}

\author[D. Lond]{Daniel Lond}
\address
{Weta Digital Ltd.,
P.O.\ Box 15208,
9--11 Manuka Street,
Miramar,
Wellington,
New Zealand}
\email{dlond@wetafx.co.nz}

\author[B.\ Martin]{Benjamin Martin}
\address
{Department of Mathematics,
University of Aberdeen,
King's College,
Fraser Noble Building,
Aberdeen AB24 3UE,
United Kingdom}
\email{B.Martin@abdn.ac.uk}

\begin{document}

\begin{abstract}
 Let $G$ be a linear algebraic group over an algebraically closed field of characteristic $p\geq 0$.  We show that if $H_1$ and $H_2$ are connected subgroups of $G$ such that $H_1$ and $H_2$ have a common maximal unipotent subgroup and $H_1/R_u(H_1)$ and $H_2/R_u(H_2)$ are semisimple, then $H_1$ and $H_2$ are $G$-conjugate.  Moreover, we show that if $H$ is a semisimple linear algebraic group with maximal unipotent subgroup $U$ then for any algebraic group homomorphism $\sigma\colon U\ra G$, there are only finitely many $G$-conjugacy classes of algebraic group homomorphisms $\rho\colon H\ra G$ such that $\rho|_U$ is $G$-conjugate to $\sigma$.  This answers an analogue for connected algebraic groups of a question of B. K\"ulshammer.
 
 In K\"ulshammer's original question, $H$ is replaced by a finite group and $U$ by a Sylow $p$-subgroup of $H$; the answer is then known to be no in general.  We obtain some results in the general case when $H$ is non-connected and has positive dimension.  Along the way, we prove existence and conjugacy results for maximal unipotent subgroups of non-connected linear algebraic groups.  When $G$ is reductive, we formulate K\"ulshammer's question and related conjugacy problems in terms of the nonabelian 1-cohomology of unipotent radicals of parabolic subgroups of $G$, and we give some applications of this cohomological approach.  In particular, we analyse the case when $G$ is a semisimple group of rank 2.
\end{abstract}

\maketitle

\section{Introduction}
\label{sec:intro}

Let $G$ be a linear algebraic group over an algebraically closed field $k$.  A fundamental problem is to describe the subgroup structure of $G$.  Much  effort has been put into doing this when $G$ is simple (see \cite{liebeckseitz0}, \cite{liebecktesterman}, \cite{stewartG2}, \cite{stewartF4}, for example).  We prove the following result concerning subgroups of an arbitrary $G$.

\begin{thm}
\label{thm:samesubgp}
 Let $H_1$ and $H_2$ be connected subgroups of $G$ such that $R_u(H_1)= R_u(H_2)$ and $H_1/R_u(H_1)$ and $H_2/R_u(H_2)$ are semisimple.  Suppose $H_1$ and $H_2$ have a common maximal unipotent subgroup $U$.  Then $H_1$ and $H_2$ are $N_G(U)$-conjugate.
\end{thm}

\noindent Here $N_G(U)$ denotes the normaliser of $U$ in $G$ and $R_u(M)$ denotes the unipotent radical of $M$.

Given another linear algebraic group $H$, we define a {\em representation} of $H$ in $G$ to be a homomorphism of algebraic groups from $H$ to $G$; we denote by ${\rm Hom}(H,G)$ the set of representations of $H$ in $G$.  We say that $\rho\in {\rm Hom}(H,G)$ is {\em faithful} if $\rho$ is injective.  If $M\leq G$ then $M$ acts on ${\rm Hom}(H,G)$ by $(m\cdot \rho)(h)= m\rho(h)m^{-1}$ for $h\in H$, $\rho\in {\rm Hom}(H,G)$ and $m\in M$; we call the orbits {\em $M$-conjugacy classes}.  The image of a representation is a subgroup, so understanding subgroups helps us to understand representations (and vice versa).  Here is a counterpart to Theorem~\ref{thm:samesubgp} in terms of representations.

\begin{thm}
\label{thm:main}
 Suppose $H$ is connected and $H/R_u(H)$ is semisimple, and let $U$ be a maximal unipotent subgroup of $H$.  If $\rho_1, \rho_2\colon H\ra G$ are representations such that $\rho_1|_U= \rho_2|_U$ then $\rho_1$ and $\rho_2$ are $C_G(\rho_1(U))$-conjugate.
\end{thm}

\noindent (Here $C_G(\rho_1(U))$ denotes the centraliser of $\rho_1(U)$ in $G$.)

These results were inspired by work of Burkhard K\"ulshammer \cite{kuls}, which we briefly discuss now.  It is well known that if $G$ is reductive, $F$ is a finite group and either ${\rm char}(k)= 0$ or ${\rm char}(k)> 0$ and $|F|$ is coprime to ${\rm char}(k)$, then ${\rm Hom}(F,G)$ is a finite union of $G$-conjugacy classes (see \cite[I.4, Thm.\ 2]{slodowy} and Lemma~\ref{lem:finitelinred}).  Now suppose ${\rm char}(k)= p> 0$.  If $p$ divides $|F|$ then simple examples show that ${\rm Hom}(F,G)$ can contain infinitely many $G$-conjugacy classes (see \cite[Sec.\ 1]{BMR_kuls}, for example).  To obtain a useful finiteness result, one needs to impose extra restrictions.  Let $F_p$ be a Sylow $p$-subgroup of $F$.  K\"ulshammer asked whether there are only finitely many $G$-conjugacy classes of representations $\rho\in {\rm Hom}(F,G)$ such that $\rho|_{F_p}$ lies in a fixed $G$-conjugacy class \cite[Sec.\ 2]{kuls}.  We give a version of this question that applies to an arbitrary linear algebraic group $H$.

\begin{qn}
\label{qn:algKQ}
 Let $U$ be a maximal unipotent subgroup of $H$.  Is it true that for all $\sigma\in {\rm Hom}(U,G)$, there are only finitely many $G$-conjugacy classes of representations $\rho\in {\rm Hom}(H,G)$ such that $\rho|_U$ is $G$-conjugate to $\sigma$?
\end{qn}

\noindent Theorem~\ref{thm:main} shows that the answer is yes if $H$ is semisimple---in fact, in this case $\rho$ is unique up to $G$-conjugacy if it exists.  Note that maximal unipotent subgroups of $H$ exist and are unique up to conjugacy; this is well known when $H$ is connected, and we give a proof below in the non-connected case (Proposition~\ref{prop:nonconnmaxunipt}).  Because of this, it is easily seen that the answer to Question~\ref{qn:algKQ} for a given pair $(G,H)$ does not depend on the choice of $U$.  If $H$ is finite and ${\rm char}(k)= p> 0$ then maximal unipotent subgroups of $H$ are the same as Sylow $p$-subgroups of $H$ (see Proposition~\ref{prop:nonconnmaxunipt}(a)), so we recover K\"ulshammer's original question.  Our formulation of the question makes sense in characteristic 0 as well.

Assume for the rest of the paragraph that $H$ is finite and $p> 0$.  K\"ulshammer proved using a straightforward representation-theoretic argument that the answer to Question~\ref{qn:algKQ} is yes when $G= {\rm GL}_n(k)$ \cite[Sec.\ 2]{kuls}.  Slodowy showed that the answer is yes for connected reductive $G$ when $p$ is a good prime for $G$ \cite[I.5, Thm.\ 3]{slodowy}: one embeds $G$ in some ${\rm GL}_n(k)$ and studies the behaviour of the induced map ${\rm Hom}(H,G)\ra {\rm Hom}(H,{\rm GL}_n(k))$, applying a celebrated geometric argument of Richardson \cite[Sec.\ 3]{rich2}.  In particular, the answer is yes for any $p> 0$ if every simple component of $G$ is of type $A$.  On the other hand, an example of Cram shows that the answer is no for $H= S_3$, $p= 2$ and $G$ a certain 3-dimensional non-connected group with $G^0$ unipotent \cite{cram}.  Bate, R\"ohrle and the second author recently gave an example for $G$ simple of type $G_2$ in characteristic 2 for which the answer is no \cite{BMR_kuls}.  Uchiyama has constructed further such examples for $G$ of type $E_6$, $E_7$ and $E_8$ in characteristic 2 \cite[Sec.\ 3]{uchiyama2}, \cite[Sec.\ 6.1]{uchiyama3}.

Now suppose $H$ is connected and positive-dimensional.  If $H$ has a nontrivial torus as a quotient and $G$ contains a nontrivial torus $S$ then the answer to Question~\ref{qn:algKQ} is no.  For just take a nontrivial representation $\rho\colon H\ra S$; it is easily seen that the representations $\rho_n\colon H\ra G$ defined by $\rho_n(h)= \rho(h)^n$ for $n\in {\mathbb N}$ are pairwise non-$G$-conjugate.  Note that $\rho_n|_U$ is the trivial representation for each $n$, so the conclusion of Theorem~\ref{thm:main} fails.  This is the reason for the semisimplicity hypothesis in Theorems~\ref{thm:samesubgp} and \ref{thm:main} (and elsewhere in the paper).  By a similar argument, if $G$ contains a torus of dimension at least 2 then Theorem~\ref{thm:samesubgp} can fail without the semisimplicity assumption on $H^0$.  But under suitable hypotheses, Theorem~\ref{thm:main} is a stepping stone which lets us extend results from the case when $H$ is finite to the case when $H$ has positive dimension (cf.\ the paragraph following Theorem~\ref{thm:GcrKQ} below).

The proofs of Theorems~\ref{thm:samesubgp} and \ref{thm:main} are quite short; they are based on geometric invariant theory and standard structure theory of linear algebraic groups.  In some of the subsequent results, the theory of $G$-complete reducibility is important.  Recall \cite{BMR} that if $G$ is connected and reductive then a subgroup $H$ of $G$ is said to be {\em $G$-completely reducible} ($G$-cr) if whenever $H$ is contained in a parabolic subgroup $P$ of $G$, $H$ is contained in a Levi subgroup of $P$ (see Section~\ref{sec:prelim} for the definition when $G$ is non-connected).  We say a representation of $H$ in $G$ is $G$-cr if its image is $G$-cr.  We obtain a result for non-connected groups as well if we restrict ourselves to $G$-cr representations:

\begin{thm}
\label{thm:GcrKQ}
 Let $G$ be reductive, and suppose $H^0/R_u(H)$ is semisimple.  Let $U$ be a maximal unipotent subgroup of $H$ and let $\sigma\in {\rm Hom}(U,G)$.  Then there are only finitely many $G$-conjugacy classes of $G$-cr representations $\rho\in {\rm Hom}(H,G)$ such that $\rho|_U$ is $G$-conjugate to $\sigma$.
\end{thm}

\noindent In the special case when $H$ is finite and $G$ is reductive, Theorem~\ref{thm:cr_finite} shows that there are only finitely many $G$-conjugacy classes of $G$-cr representations of $H$ in $G$.  The proof of Theorem~\ref{thm:GcrKQ} rests on an argument that combines this special case with Theorem~\ref{thm:main}.  A similar argument also allows us to settle the characteristic 0 case:

\begin{thm}
\label{thm:char0KQ}
 Suppose ${\rm char}(k)= 0$ and $H^0/R_u(H)$ is semisimple.  Then the answer to Question~\ref{qn:algKQ} is yes for $H$.
\end{thm}

An important tool for studying subgroups of, and representations into, a reductive group $G$ is nonabelian 1-cohomology.  Let $\rho\in {\rm Hom}(H,G)$ and let $P$ be a parabolic subgroup of $G$ such that $\rho(H)\subseteq P$.  Then
representations near $\rho$ in an appropriate sense can be understood 
in terms of a certain nonabelian 1-cohomology $H^1(H,V)$, where $V= R_u(P)$ (see Section~\ref{sec:H1}).  In particular, if $H^1(H,V)$ vanishes then $\rho(H)$ is $V$-conjugate to a subgroup of a Levi subgroup $L$ of $P$.  Liebeck and Seitz used this idea to prove results about $G$-complete reducibility for $G$ simple and of exceptional type when $p$ is not too small \cite{liebeckseitz0}.  Stewart investigated \mbox{(non-)$G$-completely} reducible subgroups for small $p$ \cite{stewartG2}, \cite{stewartF4} and proved some general results about the behaviour of the first- and higher nonabelian cohomologies of $V$ \cite{stewartTAMS}.

In our setting we have an extra ingredient: restricting $\rho$ to a maximal unipotent subgroup $U$ of $H$ gives rise to a map of 1-cohomologies $H^1(H,V)\ra H^1(U,V)$, and the fibres of this map give us information relevant to Question~\ref{qn:algKQ}.  In Section~\ref{sec:H1} we study this construction and give a cohomological criterion (Theorem~\ref{main_thm}) which in some cases helps to show that Question~\ref{qn:algKQ} has positive answer---see Section~\ref{sec:appln} and Theorem~\ref{thm:lowrank}.

The paper is set out as follows.  In Section~\ref{sec:prelim} we give some preliminary results on algebraic groups and their actions, and in Section~\ref{sec:maxlunipt} we study maximal unipotent subgroups of non-connected groups.  We prove Theorems~\ref{thm:samesubgp}, \ref{thm:main}, \ref{thm:GcrKQ} and \ref{thm:char0KQ} in Section~\ref{sec:mainproof} (for the latter two, see Theorems~\ref{thm:GcrKQalt} and \ref{thm:linredKQ}, respectively).  In Section~\ref{sec:H1} we describe our cohomological approach and in Section~\ref{sec:appln} we give some applications of it.  In Section~\ref{sec:rank2} we study groups of semisimple rank 2.


\bigskip
\noindent {\bf Acknowledgments}:
Some of the work in this paper was carried out by the first author during his PhD \cite{lond}.  Both authors acknowledge the financial support of Marsden Grants UOC0501, UOC1009 and UOA1021.  We are grateful to Dave Benson and G\"unter Steinke for helpful conversations.  We also thank the referee for their careful reading of the paper.

\section{Preliminaries}
\label{sec:prelim}

We fix an algebraically closed field $k$ of characteristic $p\geq 0$.  All varieties and algebraic groups are defined over $k$ and are affine unless otherwise stated; in particular, all algebraic groups are linear.  By a subgroup of an algebraic group we mean a closed subgroup, and homomorphisms of algebraic groups are understood to be morphisms of varieties.  We assume $G$ and $H$ are possibly non-connected algebraic groups over $k$.  We allow reductive algebraic groups to be non-connected, but we take simple and semisimple algebraic groups to be connected by definition.  If $h\in H$ then we denote by $h_s$ and $h_u$ the semisimple and unipotent part of $h$, respectively.  Given $A_1, A_2\subseteq H$, we write $A_1A_2$ for the product $\{a_1a_2\mid a_1\in A_1, a_2\in A_2\}$.  If $m\in {\mathbb N}$ then we denote by $C_m$ the cyclic group of order $m$ and by $D_{2m}$ the dihedral group of order $2m$.

By an {\em action} of $H$ on a variety $X$, we mean a morphism of varieties $H\times X\ra X$ that is a left action of $H$ on $X$.  Given such an action and given $x\in X$, we denote by $H\cdot x$ the orbit of $x$ and by $H_x$ the stabiliser of $x$.

Recall that $H$ is said to be {\em linearly reductive} if every rational representation of $H$ is completely reducible. If $p= 0$ then $H$ is linearly reductive if and only if $H$ is reductive, while if $p> 0$ then $H$ is linearly reductive if and only if every element of $H$ is semisimple if and only if $H^0$ is a torus and $|H:H^0|$ is coprime to $p$ (see \cite{nagata}).

If $G$ is reductive, $T$ is a maximal torus of $G$ and $M$ is a $T$-stable subgroup of $G$ then we denote by $\Phi_T(M)$ the set of roots of $M$ with respect to $T$.  If $\alpha\in \Phi$ then we denote by $U_\alpha$ the corresponding root group and by $G_\alpha$ the rank 1 semisimple group $\langle U_\alpha\cup U_{-\alpha}\rangle$.

To simplify the statement of our results, we adopt the following convention: if $p= 0$ then by a finite $p$-group we mean the trivial group, and by a Sylow $p$-subgroup of a finite group we mean the trivial subgroup.  Note that unipotent groups are always connected in characteristic 0 (cf.\ \cite[Ex.\ 15.11]{hum}).

By a maximal unipotent subgroup of $H$, we mean a unipotent subgroup $U$---not necessarily proper---that is maximal with respect to inclusion (so $U= H$ if $H$ is unipotent).  If $H$ is connected then the structure of maximal unipotent subgroups is well known: it follows from \cite[30.4]{hum} that every unipotent subgroup of $H$ is contained in a Borel subgroup of $H$ (this is proved for reductive $H$ in {\em loc.\ cit.}, but the general case follows easily).  It now follows from \cite[19.3 Thm.(a) and 21.3, Cor.\ A]{hum} that the maximal unipotent subgroups of connected $H$ are precisely the unipotent radicals of the Borel subgroups of $H$, they are unique up to conjugacy and they are connected; moreover, we see that every unipotent subgroup of $H$ is contained in a maximal unipotent subgroup.  In Section~\ref{sec:maxlunipt}, we establish analogous results for non-connected $H$.

\begin{defn}
 We say that $(G,H)$ is a {\em K\"ulshammer pair} if Question~\ref{qn:algKQ} has positive answer for $G$ and $H$.  We say that $G$ {\em has the K\"ulshammer property} if $(G,H)$ is a K\"ulshammer pair for every $H$ such that $H^0/R_u(H)$ is semisimple.
\end{defn}

\begin{rem}
\label{rem:abstractiso}
 If $H$ is finite then any function $f\colon H\ra G$ is automatically a morphism of varieties.  It follows in this case that if two algebraic groups $G_1$ and $G_2$ are isomorphic as abstract groups then $(G_1, H)$ is a K\"ulshammer pair if and only if $(G_2, H)$ is a K\"ulshammer pair.
\end{rem}

The following result is immediate.

\begin{lem}
\label{lem:prod}
 Let $G_1, G_2$ be algebraic groups.  Then $(G_1\times G_2, H)$ is a K\"ulshammer pair if and only if $(G_1, H)$ and $(G_2, H)$ are K\"ulshammer pairs.
\end{lem}

Suppose $H$ is connected, let $B$ be a Borel subgroup of $H$, let $X$ be an affine variety and let $f\colon H\ra X$ be a morphism such that $f(hb)= f(h)$ for all $h\in H$ and all $b\in B$.  Then $H/B$ is projective \cite[21.3\ Thm.]{hum} and $f$ gives rise to a morphism $\ovl{f}$ from $H/B$ to $X$.  Since $H/B$ is connected and $X$ is affine, $\ovl{f}$ must be constant, so $f$ is constant.  In particular, if $V$ is an affine $H$-variety, $v\in V$ and the stabiliser $H_v$ contains $B$ then $H_v= H$: to see this, just apply the argument immediately above to the orbit map $f\colon H\ra V$, $h\mapsto h\cdot v$.

\begin{lem}
\label{lem:common_Borel}
 Let $H_1, H_2$ be connected reductive subgroups of $G$.  Suppose $B$ is a common Borel subgroup of both $H_1$ and $H_2$.  Then $H_1= H_2$.
\end{lem}

\begin{proof}
 The quotient variety $G/H_1$ is affine since $G$ is an affine variety and $H_1$ is reductive, and $H_2$ acts on $G/H_1$ by left multiplication.  The stabiliser in $H_2$ of the coset $H_1$ is $H_1\cap H_2$, which contains $B$, so it must equal the whole of $H_2$.  Hence $H_2\subseteq H_1$.  The reverse inequality follows similarly, so $H_1= H_2$.
\end{proof}

\noindent Here is the corresponding result for representations.

\begin{lem}
\label{lem:Borel_det}
 Suppose $H$ is connected and let $B$ be a Borel subgroup of $H$.  Let $\rho_1,\rho_2\in {\rm Hom}(H,G)$ such that $\rho_1|_B= \rho_2|_B$.  Then $\rho_1= \rho_2$.
\end{lem}

\begin{proof}
 Define $f\colon H\ra G$ by $f(h)= \rho_1(h)\rho_2(h)^{-1}$.  For any $h\in H$ and any $b\in B$, $f(hb)= \rho_1(hb)\rho_2(hb)^{-1}= \rho_1(h)\rho_1(b)\rho_2(b)^{-1}\rho_2(h)^{-1}=  \rho_1(h)\rho_2(h)^{-1}= f(h)$.  So $f$ is constant with value $f(1)= 1$, and the result follows.
\end{proof}

\begin{lem}
\label{lem:redsub}
 Suppose $H$ is semisimple and $U$ is a maximal unipotent subgroup of $H$.  Then the only reductive subgroup of $H$ that contains $U$ is $H$.
\end{lem}

\begin{proof}
 Let $M$ be a reductive subgroup of $H$ containing $U$.  As $U$ is connected, it is enough to prove the result under the extra hypothesis that $M$ is connected, so we shall assume this.  Now $U$ is a maximal unipotent subgroup of $M$, so there is a maximal torus $S$ of $M$ such that $SU$ is a Borel subgroup of $M$; in particular, $S$ normalises $U$.  Choose a maximal torus $T$ of $N_H(U)$ such that $S\leq T$; then $T$ is a maximal torus of $H$, $TU$ is a Borel subgroup of $H$ and $U$ contains all of the positive root groups of $H$ with respect to the pair $(B,T)$.  Now $T$ normalises the Borel subgroup $SU$ of $M$, so $T$ normalises $M$ by Lemma~\ref{lem:common_Borel}.  As $T$ normalises both $SU$ and $S$, $T$ must normalise the unique unipotent subgroup $U^-$ of $M$ that is opposite to $U$ with respect to $S$.  As $U^-$ is $M$-conjugate to $U$, $U^-$ is also a maximal unipotent subgroup of $H$.  We see that $U^-$ is the unique unipotent subgroup of $H$ that is opposite to $U$ with respect to $T$; in particular, $U^-$ contains all of the negative root groups of $H$ with respect to the pair $(B,T)$.  It follows from \cite[27.5 Thm.(e)]{hum} that $M= H$, as required.
\end{proof}

We now briefly recall the theory of $G$-completely reducible subgroups of a reductive group \cite{serre1}, \cite{serre2}.  Assume $G$ is reductive until the end of this section.  We need the notion of R-parabolic and R-Levi subgroups of $G$ (see \cite[Sec.\ 6]{BMR} for definitions and further details).  Let $f$ be a morphism from $k^*$ to a (not necessarily affine) variety $X$.  We say that {\em $\lim_{a\to 0} f(a)$ exists} if $f$ extends to a (necessarily unique) morphism $\widehat{f}\colon k\ra X$; in this case, we write $\lim_{a\to 0} f(a)= \widehat{f}(0)$.  We write $Y(G)$ for the set of cocharacters of $G$.  Given $\lambda\in Y(G)$, we define $P_\lambda= \{g\in G\mid \lim_{a\to 0} \lambda(a)g\lambda(a)^{-1} \ \mbox{exists}\}$; then $P_\lambda\leq G$ and we call a subgroup of this form an {\em R-parabolic subgroup} of $G$.  We define $L_\lambda= C_G(\lambda(k^*))$, and we call $L_\lambda$ an {\em R-Levi subgroup} of $P_\lambda$; then $P_\lambda= L_\lambda\ltimes R_u(P_\lambda)$.  We denote by $c_\lambda$ the canonical projection from $P_\lambda$ onto $L_\lambda$; we have $c_\lambda(g)= \lim_{a\to 0} \lambda(a)g\lambda(a)^{-1}$ for all $g\in P_\lambda$.  In particular, $R_u(P_\lambda)= \{g\in P_\lambda\mid \lim_{a\to 0} \lambda(a)g\lambda(a)^{-1}= 1\}$.  If $P$ is an R-parabolic subgroup of $G$ then any two R-Levi subgroups of $P$ are $P$-conjugate.

If $G$ is connected then R-parabolic subgroups and R-Levi subgroups correspond to parabolic subgroups and Levi subgroups in the usual sense; for non-connected $G$, if $P$ is an R-parabolic subgroup then $P\cap G^0$ is a parabolic subgroup of $G$.  Any R-parabolic subgroup is contained in a maximal R-parabolic subgroup, and there are only finitely many $G$-conjugacy classes of R-parabolic subgroups.

A subgroup $M$ of $G$ is {\em $G$-completely reducible} ($G$-cr) if whenever $M$ is contained in an R-parabolic subgroup $P$ of $G$, there is an R-Levi subgroup $L$ of $P$ such that $M\leq L$; $M$ is {\em $G$-irreducible} ($G$-ir) if $M$ is not contained in any proper parabolic subgroup of $G$.  Clearly a $G$-ir subgroup is $G$-cr.  If $G= {\rm SL}_n(k)$ or ${\rm GL}_n(k)$ then $M$ is $G$-cr (resp.\ $G$-ir) if and only if the inclusion $M\ra G$ is a completely reducible (resp.\ irreducible) representation of $M$ in the usual sense.  A $G$-cr subgroup is reductive, and any linearly reductive subgroup of $G$ is $G$-cr; in particular, if $p= 0$ then a subgroup of $G$ is $G$-cr if and only if it is reductive.  If $p> 0$, however, then there can exist reductive subgroups of $G$ that are not $G$-cr.  If $G^0$ is a torus then $R_u(P_\lambda)= 1$ and so $P_\lambda= L_\lambda$ for any $\lambda\in Y(G)$; it follows in this case that every subgroup of $G$ is $G$-cr.

If $\rho\in {\rm Hom}(H,G)$ then we say that $\rho$ is $G$-cr if $\rho(H)$ is $G$-cr.  We define
$$ {\rm Hom}(H,G)_{\rm cr}= \{\rho\in {\rm Hom}(H,G)\mid \rho\ \mbox{is $G$-cr}\} $$
and
$$ {\rm Hom}(H,G)_{\rm ir}= \{\rho\in {\rm Hom}(H,G)\mid \rho\ \mbox{is $G$-ir}\}. $$

We recall a useful result.

\begin{thm}[{\cite[Cor.\ 3.8 and Sec.\ 6]{BMR}}]
\label{thm:cr_finite}
 Let $F$ be a finite group.  Then ${\rm Hom}(F,G)_{\rm cr}$ is a finite union of $G$-conjugacy classes.
\end{thm}

\section{Maximal unipotent subgroups in non-connected groups}
\label{sec:maxlunipt}

In this section we establish some results on maximal unipotent subgroups of non-connected groups.

\begin{lem}
\label{lem:uniptextnquot}
 \begin{itemize}
 \item[(a)] An extension of unipotent groups is unipotent.
 \item[(b)]  Let $N$ be a unipotent normal subgroup of $H$ and let $\pi_N\colon H\ra H/N$ be the canonical projection.  Then for all $U\leq H$, $U$ is a maximal unipotent subgroup of $H$ if and only if $N\leq U$ and $\pi_N(U)$ is a maximal unipotent subgroup of $H/N$.
 \end{itemize}
\end{lem}

\begin{proof}
 (a) Let $1\ra N\stackrel{i}{\ra} M\stackrel{q}{\ra} Q\ra 1$ be a short exact sequence of algebraic groups such that $N$ and $Q$ are unipotent.  Let $m\in M$.  Since $Q$ is unipotent, $q(m_s)= 1$, so $m_s\in N$.  But $N$ is unipotent, so $m_s$ must be trivial.  Hence $m$ is unipotent.\smallskip\\
 (b) Suppose $U$ is a maximal unipotent subgroup of $H$.  If $M$ is a unipotent subgroup of $H/N$ with $\pi_N(U)\leq M$ then $(\pi_N)^{-1}(M)$ is unipotent by part (a), so $(\pi_N)^{-1}(M)= U$.  If $M$ properly contains $\pi_N(U)$ then $(\pi_N)^{-1}(M)$ properly contains $U$, a contradiction.  It follows that $\pi_N(U)$ is a maximal unipotent subgroup of $H/N$.  Moreover, we see (taking $M= \pi_N(U)$) that $N\leq U$.
 
 Conversely, suppose $N\leq U$ and $\pi_N(U)$ is a maximal unipotent subgroup of $G/N$.  If $M$ is a unipotent subgroup of $H$ that properly contains $U$ then $\pi_N(M)$ properly contains $\pi_N(U)$, contradicting the maximality of $\pi_N(U)$.  We deduce that $U$ is a maximal unipotent subgroup of $H$, as required.
\end{proof}

We can now state the main result of this section.

\begin{prop}
\label{prop:nonconnmaxunipt}
 \begin{itemize}
 \item[(a)] Let $U\leq H$.  Then $U$ is a maximal unipotent subgroup of $H$ if and only if $U^0$ is a maximal unipotent subgroup of $H^0$ and $U/(U\cap H^0)$ is a Sylow $p$-subgroup of $H/H^0$.
 \item[(b)] Every unipotent subgroup of $H$ is contained in a maximal unipotent subgroup of $H$.  In particular, maximal unipotent subgroups of $H$ exist.
 \item[(c)] Maximal unipotent subgroups of $H$ are unique up to conjugacy.
\end{itemize}
\end{prop}

\begin{lem}
\label{lem:onedirn}
 The implication $\Longleftarrow$ holds in part (a) of Proposition~\ref{prop:nonconnmaxunipt}.
\end{lem}

\begin{proof}
 Suppose $U^0$ is a maximal unipotent subgroup of $H^0$ and $U/(U\cap H^0)$ is a Sylow $p$-subgroup of $H/H^0$.  Then the subgroups $U\cap H^0= U^0$ and $U/(U\cap H^0)$ are unipotent, so $U$ is unipotent by Lemma~\ref{lem:uniptextnquot}.  Now suppose $V\leq H$ is unipotent and $U\leq V$.  Then $V\cap H^0= U^0$ by the maximality of $U^0$ in $H^0$, so we may identify $U/(U\cap H^0)$ with a subgroup of $V/(V\cap H^0)$ inside $H/H^0$.  Now $V/(V\cap H^0)$, being finite and unipotent, is a finite $p$-group; but $U/(U\cap H^0)$ is a Sylow $p$-subgroup of $H/H^0$, so we must have $V/(V\cap H^0)= U/(U\cap H^0)$.  It follows that $V= U$.  Hence $U$ is a maximal unipotent subgroup of $H$.
\end{proof}

Next we prove a special case of Proposition~\ref{prop:nonconnmaxunipt}.

\begin{lem}
\label{lem:torusidcpt}
 Proposition~\ref{prop:nonconnmaxunipt} holds if $H^0$ is a torus.
\end{lem}

\begin{proof}
 If ${\rm char}(k)= 0$ then any unipotent group is connected, so the lemma is true by the results on maximal unipotent subgroups of connected groups.  So suppose ${\rm char}(k)= p>0$.  Set $t= |H/H^0|$.  If $U$ is a unipotent subgroup of $H$ then $U\cap H^0= 1$, so the canonical projection $\psi\colon H\ra H/H^0$ maps $U$ injectively to $H/H^0$.  Hence $|U|\leq t$.  Part (b) now follows immediately.
 
 We now prove that any two maximal unipotent subgroups of $H$ are conjugate.  Let $k_0$ be the algebraic closure of ${\mathbb F}_p$; we wish to reduce to the case when $k= k_0$.  We do this as follows.  By \cite[Prop.\ 3.2]{martin}, $H$ admits a $k_0$-structure.
  Let $F_1,\ldots, F_r$ be representatives of the isomorphism classes of finite $p$-groups of order at most $t$.  For each $i$, we can choose $\gamma_1^{(i)}, \ldots, \gamma_t^{(i)}\in F_i$ such that the $\gamma_j^{(i)}$ generate $F_i$.  Let $C_i= {\rm Hom}(F_i,H)$; we can identify $C_i$ with a closed $k_0$-defined subvariety of $H^t$ via the map $C_i\ra H^t$, $\rho\mapsto (\rho(\gamma_1^{(i)}),\ldots, \rho(\gamma_t^{(i)}))$.  We let $H$ act on $H^t$ by simultaneous conjugation \cite[Sec.\ 1]{BMR}.  Each subset $C_i$ is $H$-stable.  Let $C= \bigcup_{i= 1}^r C_i$.  Then $C$ is $k_0$-defined.
 
 Every subgroup of $H$ is $H$-cr as $H^0$ is a torus.  It follows from \cite[Cor.\ 3.7 and Sec.\ 6]{BMR} that $H\cdot (h_1,\ldots, h_t)$ is closed for every $(h_1,\ldots, h_t)\in H^t$.  By Theorem~\ref{thm:cr_finite}, each $C_i$ is a finite union of $H$-conjugacy classes.  Hence $C$ is a union of finitely many $H^0$-conjugacy classes, each of which is closed.  This means that these $H^0$-conjugacy classes are precisely the irreducible components of $C$, so each such class is $k_0$-defined.  In particular, each $H^0$-conjugacy class in $C$ contains a $k_0$-point.
 
 So let $U_1$ and $U_2$ be maximal unipotent subgroups of $H$.  Then $U_1= {\rm Im}(\rho)$ for some $i$ and some $\rho\in C_i$.  Hence we can assume by the previous paragraph---after conjugating $U_1$ by some element of $H^0$ if necessary---that $U_1\leq H(k_0)$.  Likewise, we can assume that $U_2\leq H(k_0)$.  There is an ascending sequence $H_1\leq H_2\leq \cdots$ of finite subgroups of $H(k_0)$ such that $\ds H(k_0)= \bigcup_{m= 1}^\infty H_m$ (see Remark~\ref{rem:reduction} below).  Since $U_1$ and $U_2$ are finite, there exists $m\in {\mathbb N}$ such that $U_1, U_2\leq H_m$.  Then $U_1$ and $U_2$ are maximal unipotent subgroups of the finite group $H_m$, so $U_1$ and $U_2$ are Sylow $p$-subgroups of $H_m$ and hence are $H_m$-conjugate to each other.  This proves part (c).
 
To finish, we prove part (a).  By the proof of \cite[Prop.\ 3.2]{martin}, there is a finite subgroup $F$ of $H$ such that $H= FH^0$.  Let $F_p$ be a Sylow $p$-subgroup of $F$.  It is easily seen that $\psi(F_p)$ is a Sylow $p$-subgroup of $H/H^0$, so $F_p$ is a maximal unipotent subgroup of $H$ by Lemma~\ref{lem:onedirn}.  Part (a) now follows from part (c) and Lemma~\ref{lem:onedirn}.
\end{proof}

\begin{lem}
\label{lem:uniptidcpt}
 Suppose $H^0$ is not a torus, and let $M$ be any unipotent subgroup of $H$. Then there is a unipotent subgroup $U$ of $H$ such that $M\leq U$ and $U\cap H^0$ is nontrivial.
\end{lem}

\begin{proof}
 If ${\rm char}(k)= 0$ then any unipotent group is connected, and the result is immediate.  So suppose ${\rm char}(k)= p>0$. If $M$ is trivial then we can take $U$ to be any nontrivial unipotent subgroup of $H^0$, so assume $M$ is nontrivial.  If ${\rm dim}(M)> 0$ then $M\cap H^0$ is nonempty, so we can take $U= M$; hence we can assume without loss that $M$ is finite.  We use induction on ${\rm dim}(H)$.  Clearly if ${\rm dim}(H)= 0$ then there is nothing to prove, since this case cannot occur under our assumption on $H^0$.  So let ${\rm dim}(H)= n> 0$ and suppose the result holds for any group of dimension less than $n$.  If $H$ is non-reductive then we can take $U$ to be $MR_u(H)$ by Lemma~\ref{lem:uniptextnquot}(a), so without loss we assume $H$ is reductive; in particular, $H_1:= [H^0, H^0]$ is nontrivial and semisimple.  If $M$ centralises $H_1$ then we can take $U$ to be $MM_1$ by Lemma~\ref{lem:uniptextnquot}(a), where $M_1$ is any nontrivial unipotent subgroup of $H_1$, and we are done.  We assume, therefore, that $Z:= C_M(H_1)$ is a proper subgroup of $M$.  Clearly $Z$ is normal in $M$.
 
 As $M/Z$ is nontrivial, we can choose $m\in M$ such that the image of $m$ in $M/Z$ belongs to the centre $Z(M/Z)$ and has order $p$.  Let $\phi\in {\rm Aut}(H_1)$ be conjugation by $m$; note that $\phi$ has order $p$.  Consider the centraliser $N_1:= C_{H_1}(m)^0= C_{H_1}(\phi)^0$.  By construction, $N_1$ is an $M$-stable proper subgroup of $H_1$, so $MN_1$ has dimension less than $n$ (recall that $M$ is finite).  It is enough to prove that $N_1$ contains a nontrivial unipotent element---for then we are done by our induction hypothesis applied to $MN_1$.
  
 First suppose $\phi$ stabilises every simple component of $H_1$.  Fix a simple component $A$ of $H_1$.  Then $C_A(\phi)^0$ contains nontrivial unipotent elements; this is clear if $\phi$ is an inner automorphism of $A$ since $\phi$ is unipotent, while if $\phi$ is an outer automorphism of $A$ then it follows from \cite[Rem.~2.9]{liebeckseitz1}.  Now suppose $\phi$ does not stabilise every simple component of $H_1$.  Then there are simple components $A_1,\ldots, A_p$ of $H_1$ that are cyclically permuted by $\phi$.  Let $V$ be a nontrivial connected unipotent subgroup of $A_1$: then $\phi$ centralises the nontrivial connected unipotent subgroup $\{u\phi(u)\phi^2(u)\cdots \phi^{p-1}(u)\mid u\in V\}$ of $H_1$.  In both cases, $N_1$ contains nontrivial unipotent elements, so we are done.
\end{proof}

\begin{proof}[Proof of Proposition~\ref{prop:nonconnmaxunipt}] 
 Suppose $H$ has a nontrivial connected unipotent normal subgroup $N$; let $\pi_N\colon H\ra H/N$ be the canonical projection.  By Lemma~\ref{lem:uniptextnquot}(b), a unipotent subgroup $M$ of $H$ (resp.\ $H^0$) is maximal in $H$ (resp.\ $H^0$) if and only if $M\geq N$ and $\pi_N(M)$ is maximal in $H/N$ (resp.\ $H^0/N$), and if $V$ is a maximal unipotent subgroup of $H$ then $V= M$ if and only if $\pi_N(V)= \pi_N(M)$.  The map from $H/H^0$ to $(H/N)/(H/N)^0$ induced by $\pi_N$ is an isomorphism of finite groups.  It now follows that parts (a)--(c) hold for $H$ if they hold for $H/N$.
 
 To complete the proof, we use induction on ${\rm dim}(H)$.  By the preceding paragraph, we can assume $H$ is reductive.  We have shown the result holds if $H^0$ is a torus (Lemma~\ref{lem:torusidcpt})---in particular, it holds if ${\rm dim}(H)= 0$.  So suppose $H^0$ is not a torus.  Let $U$ be a unipotent subgroup of $H$.  By Lemma~\ref{lem:uniptidcpt}, there is a unipotent subgroup $U_1$ of $H$ such that $U\leq U_1$ and $U_1\cap H^0$ is nontrivial.  The well-known Borel-Tits construction \cite[Sec.\ 30.3]{hum} yields a parabolic subgroup $P$ of $H^0$ such that $U_1\leq N_H(P)$ and $U_1\cap H^0\leq R_u(P)$.  The latter condition implies that $P$ is proper in $H^0$, so ${\rm dim}(N_H(P))< {\rm dim}(H)$.  By our induction hypotheses, parts (a)--(c) hold for $N_H(P)$.  So there is a maximal unipotent subgroup $U_2$ of $N_H(P)$ such that $U_1\leq U_2$.  Set $U_0= U_2^0$.  By part (a), $U_0$ is a maximal unipotent subgroup of the connected group $N_H(P)^0= N_{H^0}(P)$.  It follows that $U_0$ is a maximal unipotent subgroup of $H^0$.
 
 Consider the group $N_H(U_0)$.  Clearly $U\leq N_H(U_0)$ and $U_0$ is a maximal unipotent subgroup of $N_H(U_0)^0= N_{H^0}(U_0)$.  As all maximal unipotent subgroups of $H^0$ are $H^0$-conjugate, $N_H(U_0)$ meets every connected component of $H$, so the induced map from $N_H(U_0)/N_H(U_0)^0$ to $H/H^0$ is a bijection.  The construction of the previous paragraph shows that if $\widetilde{U}$ is another unipotent subgroup of $H$ then $\widetilde{U}\leq N_H(\widetilde{U_0})$, where $\widetilde{U_0}$ is some maximal unipotent subgroup of $H^0$.  Since $\widetilde{U_0}$ is $H^0$-conjugate to $U_0$, $\widetilde{U}$ is $H^0$-conjugate to a subgroup of $N_H(U_0)$.  Parts (a)--(c) hold for $N_H(U_0)$ by our induction hypothesis as ${\rm dim}(N_H(U_0))< {\rm dim}(H)$.  To complete the proof, it is now enough to show that if $U_3$ is any unipotent subgroup of $N_H(U_0)$ then $U_3$ is a maximal unipotent subgroup of $H$ if and only if it is a maximal unipotent subgroup of $N_H(U_0)$.  The forward implication is immediate.  Conversely, suppose $U_3$ is a maximal unipotent subgroup of $N_H(U_0)$.  Let $U_4$ be a unipotent subgroup of $H$ such that $U_3\leq U_4$.  The argument of the preceding paragraph shows that $U_3\leq U_4\leq N_H(\widetilde{U_0})$ for some $H^0$-conjugate $\widetilde{U_0}$ of $U_0$.  The characterisation of maximal unipotent subgroups of $N_H(U_0)$ (resp.\ $N_H(\widetilde{U_0}))$ provided by (a) implies that $U_3$ is a maximal unipotent subgroup of $N_H(\widetilde{U_0})$, so we must have $U_3= U_4$.  Hence $U_3$ is a maximal unipotent subgroup of $H$.  This completes the proof.
\end{proof}
 
\begin{cor}
\label{cor:normaluniptmax}
 Let $N$ be a normal unipotent subgroup of $H$.  Then every maximal unipotent subgroup of $H$ contains $N$.
\end{cor}

\begin{proof}
 Let $U$ be a maximal unipotent subgroup of $H$.  Then some $H$-conjugate of $U$ contains $N$, by Proposition~\ref{prop:nonconnmaxunipt}(b) and (c).  As $N\unlhd H$, $U$ contains $N$.
\end{proof}

\begin{prop}
\label{prop:epi}
 Let $\phi\colon H\ra M$ be an epimorphism of algebraic groups and let $U$ be a maximal unipotent subgroup of $H$.  Then $\phi(U)$ is a maximal unipotent subgroup of $M$.
\end{prop}

\begin{proof}
 Since $\phi(U)$ is unipotent, there is a maximal unipotent subgroup $V$ of $M$ such that $\phi(U)\leq V$ (Proposition~\ref{prop:nonconnmaxunipt}(b)).  Replacing $M$ with $V$ and $H$ with $\phi^{-1}(V)$ if necessary, we can assume without loss that $M$ is unipotent.  We use induction on the nilpotency class of $M$ to prove that $\phi(U)= M$.  First suppose $M$ is abelian.  Let $m\in M$.  There exists $h\in H$ such that $\phi(h)= m$.  Then $\phi(h_s)= 1$, so $m= \phi(h_u)$.  By Proposition~\ref{prop:nonconnmaxunipt}(c), there exist $u\in U$ and $h_1\in H$ such that $u= h_1h_uh_1^{-1}$.   Then $\phi(u)= \phi(h_u)= m$ as $M$ is abelian, so we are done.
 
 Now assume $M$ is an arbitrary unipotent group.  Let $H_1= \phi^{-1}([M,M])$.  By our induction hypothesis, there is a maximal unipotent subgroup $U_1$ of $H_1$ such that $\phi(U_1)= [M,M]$.  We can choose a maximal unipotent subgroup $U_2$ of $H$ such that $U_1\leq U_2$ (Proposition~\ref{prop:nonconnmaxunipt}(b)).  Then $\phi(U_2)$ surjects onto $M/[M,M]$ by the abelian case above, and we deduce that $\phi(U_2)= M$.  But $U$ is $H$-conjugate to $U_2$ (Proposition~\ref{prop:nonconnmaxunipt}(c)), so $\phi(U)= M$.  The result now follows by induction.
\end{proof}

We give an application of Proposition~\ref{prop:epi} to K\"ulshammer's question.

\begin{cor}
\label{cor:unipttarget}
 Let $G$ be unipotent.  Let $U$ be a maximal unipotent subgroup of $H$ and let $\rho_1, \rho_2\in {\rm Hom}(H,G)$ such that $\rho_1|_U= \rho_2|_U$.  Then $\rho_1= \rho_2$.  In particular, $(G,H)$ is a K\"ulshammer pair.
\end{cor}

\begin{proof}
 Let $N$ be the subgroup of $H$ generated by the semisimple elements and let $H_1= H/N$.  Then $\rho_1$ and $\rho_2$ factor through $H_1$.  As $\rho_1|_U= \rho_2|_U$ and $U$ surjects onto $H_1$ (Proposition~\ref{prop:epi}), $\rho_1= \rho_2$.  The second assertion follows immediately.
\end{proof}

\section{Proof of main results}
\label{sec:mainproof}

\begin{lem}
\label{lem:persistence}
 Let $M$ be a reductive group acting on an affine variety $X$ and let $x\in X$.
 
 (a) Let $\lambda\in Y(M)$ such that $x':= \lim_{a\to 0} \lambda(a)\cdot x$ exists.  Then ${\rm dim}(M_{x'}\cap R_u(P_\lambda))\geq {\rm dim}(M_x\cap R_u(P_\lambda))$.
 
 (b) There exists $\lambda\in Y(M)$ such that $x':= \lim_{a\to 0} \lambda(a)\cdot x$ exists, $M\cdot x'$ is closed and $M_x\leq P_\lambda$.  Moreover, if $U$ is a maximal unipotent subgroup of $M_x$ and $U'$ is a maximal unipotent subgroup of $M_{x'}$ then ${\rm dim}(U')\geq {\rm dim}(U)$.
\end{lem}

\begin{proof}
 (a) Set $r= {\rm dim}(M_x\cap R_u(P_\lambda))$.  Let $S_\lambda= \lambda(k^*)R_u(P_\lambda)$, let $C= S_\lambda\cdot x$ and let $D$ be the closure of $C$; then $D$ is $R_u(P_\lambda)$-stable and $x'\in D$.
Now ${\rm dim}(R_u(P_\lambda))_y= r$ for all $y\in C$, so ${\rm dim}(R_u(P_\lambda))_y\geq r$ for all $y\in D$ by \cite[Lem.~3.7(c)]{newstead}.  Hence ${\rm dim}(M_{x'}\cap R_u(P_\lambda))\geq r$.

(b) If $M\cdot x$ is closed then we can take $\lambda= 0$ and there is nothing to prove, so suppose $M\cdot x$ is not closed.  By \cite[Cor.\ 3.5]{kempf}, there exists $\lambda\in Y(M)$ such that $x':= \lim_{a\to 0} \lambda(a)\cdot x$ exists, $M\cdot x'$ is closed and $M_x\leq P_\lambda$.  Let $U$ be a maximal unipotent subgroup of $M_x$.  By Lemma~\ref{lem:uniptextnquot}(b), $U$ contains $M_x\cap R_u(P_\lambda)$; in particular, $M_x\cap R_u(P_\lambda)= U\cap R_u(P_\lambda)$.  Let $s= {\rm dim}(U)$ and let $r= {\rm dim}(M_x\cap R_u(P_\lambda))$.  For any $u\in U$ and any $a\in k^*$, $\lambda(a)u\lambda(a)^{-1}$ belongs to $M_{\lambda(a)\cdot x}$.  It follows easily that $c_\lambda(u)$ belongs to $M_{x'}$.  Hence $c_\lambda$ gives a homomorphism from $U$ to $M_{x'}$ with kernel $M_x\cap R_u(P_\lambda)$, which implies that $c_\lambda(U)$ is an $(s-r)$-dimensional unipotent subgroup of $M_{x'}\cap L_\lambda$.  By part (a), $M_{x'}\cap R_u(P_\lambda)$ has dimension at least $r$, so the unipotent subgroup $c_\lambda(U)(M_{x'}\cap R_u(P_\lambda))$ of $M_{x'}$ has dimension at least $s$.  This completes the proof.
\end{proof}

\begin{proof}[Proof of Theorem~\ref{thm:samesubgp}]
 Since any two maximal unipotent subgroups of $H_1$ are $H_1$-conjugate, it is enough to show that $H_1$ and $H_2$ are $G$-conjugate.  First we consider the special case when $H_1$ and $H_2$ are semisimple.  Then $G/H_2$ is affine and $H_1$ acts on $G/H_2$ by left multiplication (cf.\ the proof of Lemma~\ref{lem:common_Borel}).  Let $\pi_2\colon G\ra G/H_2$ be the canonical projection and set $x= \pi_2(1)$.  By Lemma~\ref{lem:persistence}, there exists $\lambda\in Y(H_1)$ such that $x':= \lim_{a\to 0} \lambda(a)\cdot x$ exists, $H_1\cdot x'$ is closed  and $(H_1)_{x'}$ has a unipotent subgroup $U'$ of dimension at least as large as ${\rm dim}(U)$.  Since $U$ is a maximal unipotent subgroup of $H_1$, $U'$ is also.  But $(H_1)_{x'}$ is reductive by \cite[Thm.\ A]{rich}, so we deduce from Lemma~\ref{lem:redsub} that $(H_1)_{x'}= H_1$.  We can write $x'= \pi_2(g)$ for some $g\in G$; then $H_1= (H_1)_{x'}= H_1\cap gH_2g^{-1}$, so $H_1$ is $G$-conjugate to a subgroup of $H_2$.  By symmetry, $H_2$ is $G$-conjugate to a subgroup of $H_1$, so $H_2$ is $G$-conjugate to $H_1$.
 
 For the general case, let $N= R_u(H_1)= R_u(H_2)$.  Replacing $G$ with $N_G(N)$, we can assume without loss that $N\unlhd G$.  Let $\pi_N\colon G\ra G/N$ be the canonical projection.  Then $\pi_N(H_1)$ and $\pi_N(H_2)$ are semisimple, and $\pi_N(U)$ is a maximal unipotent subgroup of $\pi_N(H_1)$ and $\pi_N(H_2)$ (Proposition~\ref{prop:epi}).  By the semisimple case, $\pi_N(H_2)$ is $G/N$-conjugate to $\pi_N(H_1)$.  Since $N\leq H_1$ and $N\leq H_2$, we deduce that $H_2$ is $G$-conjugate to $H_1$, as required.
\end{proof}

\begin{rem}
 Theorem~\ref{thm:samesubgp} is false without the assumption that $R_u(H_1)= R_u(H_2)$: for instance, just take $H_1$ to be a nontrivial semisimple group and $H_2$ to be a maximal unipotent subgroup of $H_1$.
\end{rem}

\begin{ex}
 We cannot replace $N_G(U)$-conjugacy with $C_G(U)$-conjugacy in Theorem~\ref{thm:samesubgp}.  For instance, let $G= {\rm SL}_3(k)$, let $T$ be the maximal torus of diagonal matrices in $G$, let $U=\left\{\left.
 \left(
\begin{array}{ccc}
 1 & 0 & a \\
 0 & 1 & b \\
 0 & 0 & 1
\end{array}
\right)
\ \right| \ a,b\in k\right\}$ and let $H= TU$.  Set $g_c=
 \left(
\begin{array}{ccc}
 1 & c & 0 \\
 0 & 1 & 0 \\
 0 & 0 & 1
\end{array}
\right)
$ for $c\in k$.  Then the subgroups $H_c:= g_cHg_c^{-1}$ all have $U$ as a maximal unipotent subgroup as each $g_c$ normalises $U$, but a short calculation shows that $H_c$ and $H_d$ are not $C_G(U)$-conjugate unless $c= d$.
\end{ex}

\begin{proof}[Proof of Theorem~\ref{thm:main}]
 It is enough to prove the theorem when $G$ is connected, so we assume this.  We use induction on ${\rm dim}(G)$.  The result is trivial if $G= 1$.  Let $\rho_1, \rho_2\colon H\ra G$ be representations such that $\rho_1|_U= \rho_2|_U$.  Let $N= [G,G]R_u(G)$ and let $\pi_N\colon G\ra G/N$ be the canonical projection.  Note that $\rho_1(H)$ is contained in $N$: for otherwise the composition $\pi_N\circ \rho_1$ is a nontrivial homomorphism from $H$ to a torus, which is impossible as $H/R_u(H)$ is semisimple.  Likewise, $\rho_2(H)$ is contained in $N$.
 
 Let $B$ be a Borel subgroup of $H$ containing $U$ and let $T$ be a maximal torus of $B$.  Set $U'= \rho_1(U)= \rho_2(U)$.  Then $S_1:= \rho_1(T)$ and $S_2:= \rho_2(T)$ are tori of $N_G(U')$.  Let $u\in U$ and set $u'= \rho_1(u)= \rho_2(u)$.  Then for all $t\in T$,
 $$ \rho_2(t)u'\rho_2(t)^{-1}= \rho_2(tut^{-1})= \rho_1(tut^{-1})= \rho_1(t)u'\rho_1(t)^{-1}. $$
Hence there is a morphism $\mu\colon T\ra C_N(U')$ such that for all $t\in T$, $\rho_2(t)= \rho_1(t)\mu(t)$.

 Let $A= C_N(U')S_1$ (note that $S_1$ normalises $C_N(U')$ as $S_1$ normalises $N$ and $U'$).  Then $S_1, S_2\leq A$.  Choose a maximal torus $T'$ of $A$ such that $S_2\leq T'$.  By conjugacy of maximal tori of $A$, there exists $a\in A$ such that $aS_1a^{-1}\leq T'$.  Write $a= zs$ for some $z\in C_N(U')$ and some $s\in S_1$.  Set $\rho_1'= z\cdot \rho_1$ and $\rho_2'= \rho_2$.  Then $\rho_1'|_U= \rho_2'|_U$ and both $\rho_1'(T)= zS_1z^{-1}= zsS_1s^{-1}z^{-1}= aS_1a^{-1}$ and $\rho_2'(T)= S_2$ are subtori of $T'$.
Define $\mu'(t)= \rho_1'(t)^{-1}\rho_2'(t)= z\rho_1(t)^{-1} z^{-1}\rho_1(t)\mu(t)$ for $t\in T$.  Then $\mu'(t)\in C_N(U')$ for all $t\in T$ as $N_N(U')$ normalises $C_N(U')$.  Moreover, $\mu'(t)\in T'$ for all $t\in T$.  It follows that $\mu'$ is a morphism from $T$ to $T'\cap C_N(U')$. 
 
 Suppose $T'$ contains a nontrivial torus $S$ of $C_N(U')$.
Let $H$ act on $G$ by $h\cdot g= \rho_1'(h)g\rho_1'(h)^{-1}$.  If $g\in S$ then $B\leq H_g$ since $\rho_1'(B)\leq C_N(S)$, so $H_g= H$ by the paragraph following Lemma~\ref{lem:prod}.  Hence $\rho_1'(H)\leq C_N(S)^0$.
 Likewise, $\rho_2'(H)\leq C_G(S)^0$.  Let $\nu\colon N\ra N/R_u(G)$ be the canonical projection.  As $S$ is nontrivial, $\nu(S)$ is nontrivial, so $S$ is not central in the semisimple group $N/R_u(G)$.  But this implies that $C_N(S)^0$ is a proper subgroup of $N$, and hence of $G$, so $\rho_1'$ and $\rho_2'$ are $C_{C_N(S)^0}(U')$-conjugate by induction.  Thus $\rho_1$ and $\rho_2$ are $C_G(U')$-conjugate, and we're done.
 
 So suppose $T'$ does not contain a nontrivial torus $S$ of $C_N(U')$.  Then $\mu'(t)= 1$ for all $t\in T$, so $\rho_1'|_T= \rho_2'|_T$, so $\rho_1'|_B= \rho_2'|_B$.  It follows that $\rho_1'= \rho_2'$ by Lemma~\ref{lem:Borel_det}, so $\rho_1$ and $\rho_2$ are $C_G(U')$-conjugate.  This completes the proof.
\end{proof}

We deduce a useful result, which allows us to reduce to the case of finite groups (see Remark~\ref{rem:reduction}).

\begin{prop}
\label{prop:subgrpcrit}
 Let $H^0/R_u(H)$ be semisimple.  Suppose there are subgroups $H_1\leq H_2\leq\cdots$ of $H$ such that $\ds \bigcup_{m= 1}^\infty H_m$ is dense in $H$ and $(G,H_m)$ is a K\"ulshammer pair for every $m$.  Then $(G,H)$ is a K\"ulshammer pair.
\end{prop}

\begin{proof}
 By Proposition~\ref{prop:nonconnmaxunipt}(b), we can choose a maximal unipotent subgroup $U_m$ of $H_m$ for each $m$ in such a way that $U_1\leq U_2\leq\cdots$.  Set $\ds V= \ovl{\bigcup_{m= 1}^\infty U_m}$.  Then $V$ is a subgroup of $H$, and $V$ is unipotent as the set of unipotent elements of $H$ is closed, so $V$ is contained in a maximal unipotent subgroup $U$ of $H$ (Proposition~\ref{prop:nonconnmaxunipt}(b)).
 
 Let $\sigma\in {\rm Hom}(U,G)$.  Fix $\rho\in {\rm Hom}(H,G)$ such that $\rho|_U= \sigma$.  Clearly $\ds \bigcup_{m= 1}^\infty (H_m\cap H^0)$ is dense in $H^0$, so we can pick $m\in {\mathbb N}$ such that $C_G(\rho(H_m\cap H^0))= C_G(\rho(H^0))$ and $H_m$ meets every connected component of $H$.  Let $\tau\in {\rm Hom}(H,G)$ such that $\tau|_U= \sigma$.  Since $(G,H_m)$ is a K\"ulshammer pair, it is enough to show that if $\tau|_{H_m}= \rho|_{H_m}$ then $\tau= \rho$.  So suppose $\tau|_{H_m}= \rho|_{H_m}$.  By Theorem~\ref{thm:main}, there exists $g\in G$ such that $(g\cdot \tau)|_{H^0}= \rho|_{H^0}$.  So $(g\cdot \tau)|_{H_m\cap H^0}= \rho|_{H_m\cap H^0}= \tau|_{H_m\cap H^0}$, which implies that $g\in C_G(\rho(H_m\cap H^0))$.  But then $g\in C_G(\rho(H^0))$, so $\tau|_{H^0}= \rho|_{H^0}$.  As $H_m$ meets every connected component of $H$, we have $\tau= \rho$, as required.
\end{proof}

\begin{rem}
\label{rem:reduction}
 If $p> 0$ then by \cite[Lem.\ 2.3]{BMRT}, there exist finite subgroups $H_1\leq H_2\leq \cdots$ of $H$ such that $\ds \bigcup_{m= 1}^\infty H_m$ is dense in $H$.
\end{rem}

Next we prove Theorem~\ref{thm:char0KQ}.  In fact, we prove a more general version which makes sense in any characteristic.

\begin{thm}
\label{thm:linredKQ}
 Suppose $H^0/R_u(H)$ is semisimple, and suppose moreover that $H$ has a normal unipotent subgroup $N$ such that $H/N$ is linearly reductive.  Then $(G,H)$ is a K\"ulshammer pair.
\end{thm}

\begin{rem}
\label{rem:Nclassfn}
 Let $N$ be a unipotent normal subgroup of $H$.  If $H/N$ is linearly reductive then $H/N$ is reductive, so $N$ must contain $R_u(H)$.  If $p= 0$ then the converse also holds: if $N$ contains $R_u(H)$ then $H/N$ is reductive and hence linearly reductive.  In particular, if $p= 0$ and we set $N= R_u(H)$ then the hypotheses of Theorem~\ref{thm:linredKQ} hold, and we obtain Theorem~\ref{thm:char0KQ}.  On the other hand, if $p> 0$ then an algebraic group is linearly reductive if and only if all its elements are semisimple; hence the second hypothesis of Theorem~\ref{thm:linredKQ} holds if and only if the set $N'$ of unipotent elements of $H$ forms a normal subgroup of $H$, and in this case, $N$ must equal $N'$.  In particular, if $p> 0$ and $H$ is finite then $H$ has a unipotent normal subgroup $N$ such that $H/N$ is linearly reductive if and only if $H$ has a unique Sylow $p$-subgroup.
\end{rem}

First we return to the result \cite[I.4, Thm.\ 2]{slodowy} discussed in Section~\ref{sec:intro}.  A complete proof is not given in \cite{slodowy}, so we provide one for the convenience of the reader (see also Remark~\ref{rem:altpf}).

\begin{lem}
\label{lem:finitelinred}
 Let $F$ be a finite linearly reductive group.  Then ${\rm Hom}(F,G)$ is a finite union of $G$-conjugacy classes.
\end{lem}

\begin{proof}
 Choose an embedding of $G$ in some ${\rm GL}_n(k)$.  Let $\rho\in {\rm Hom}(F,G)$.  Then $\rho(F)$ is linearly reductive, so the $\rho(F)$-submodule ${\mathfrak g}= {\rm Lie}(G)$ of ${\mathfrak{gl}}_n(k)$ has a $\rho(F)$-module complement.  It follows from a slight modification of the proof of \cite[I.2, Thm.\ 1]{slodowy} that ${\rm GL}_n(k)\cdot \rho\cap {\rm Hom}(F,G)$ is a finite union of $G$-orbits.  By Maschke's Theorem, ${\rm Hom}(F,{\rm GL}_n(k))$ is a finite union of ${\rm GL}_n(k)$-conjugacy classes.  The result now follows.
\end{proof}

\noindent Lemma~\ref{lem:finitelinred} yields an immediate corollary.

\begin{cor}
 If $p= 0$ and $H$ is finite then $(G,H)$ is a K\"ulshammer pair.
\end{cor}

\begin{proof}[Proof of Theorem~\ref{thm:linredKQ}]
 First suppose $p= 0$.  Let $U$ be a maximal unipotent subgroup of $H$.  Then $U$ is connected, so $U\leq H^0$.  If $\rho_1, \rho_2\in {\rm Hom}(H,G)$ and $\rho_1|_U= \rho_2|_U$ then $\rho_1|_{H^0}$ and $\rho_2|_{H^0}$ are $C_G(U)$-conjugate by Theorem~\ref{thm:main}.  So fix $\sigma\in {\rm Hom}(H^0,G)$ and set $M= \sigma(H^0)$.  Let $C= \{\rho\in {\rm Hom}(H,G)\mid \rho|_{H^0}= \sigma\}$.  By the above discussion, it is enough to show that $C$ is a finite union of $C_G(M)$-conjugacy classes.  We can assume $C$ is nonempty, for otherwise there is nothing to prove.
 
 We claim that $H$ has a finite subgroup $F$ such that $H= FH^0$.  To see this, note that $H$ has a Levi factorisation $H= H_1\ltimes R_u(H)$ \cite[Thm.\ 7.1]{mostow}.  By the proof of \cite[Prop.\ 3.2]{martin}, we can choose a finite subgroup $F$ of $H_1$ such that $H_1= FH_1^0$; then $H= FH^0$, as required.  Let $C'= \{\rho|_F\mid \rho\in C\}$.  It is enough to show that $C'$ is a finite union of $C_G(M)$-conjugacy classes.  Fix $\rho\in C$.  Let $\rho_1\in C$.  For any $h\in H$ and any $n\in H^0$,
 
  $$ \rho_1(h)\rho_1(n)\rho_1(h)^{-1}= \rho_1(hnh^{-1})= \rho(hnh^{-1})= \rho(h)\rho(n)\rho(h)^{-1}, $$
 so $\rho(h)^{-1}\rho_1(h)$ centralises $M$.  Hence $\rho_1(F)\leq \rho(F)C_G(M)$.  As $\rho(F)$ is finite, it is enough to show that $C'$ is a finite union of $\rho(F)C_G(M)$-conjugacy classes.  But $F$ is linearly reductive because it is finite, so this follows from Lemma~\ref{lem:finitelinred}.
 
 Now suppose $p> 0$.  By Remark~\ref{rem:Nclassfn}, $N$ is the unique maximal unipotent subgroup of $H$.  By Proposition~\ref{prop:subgrpcrit} and Remark~\ref{rem:reduction}, it is enough to prove the result when $H$ is finite.  Let $\sigma\in {\rm Hom}(N,G)$, and set $M= \sigma(N)$.  Let $C= \{\rho\in {\rm Hom}(H,G)\mid \rho|_N= \sigma\}$.  It is enough to show that $C$ is a finite union of $C_G(M)$-conjugacy classes.  For all $\rho_1\in C$, $\rho_1(H)\leq N_G(M)$, so without loss we can assume that $M\unlhd G$.  As $M$ is finite, $G$ is a finite extension of $C_G(M)$ \cite[Lem.\ 6.8]{martin}, so it is enough to show that $C$ is a finite union of $G$-conjugacy classes.  Let $\pi_M\colon G\ra G/M$ be the canonical projection; set $\ovl{G}= G/M$ and $\ovl{C}= \{\pi_M\circ \rho_1\mid \rho_1\in C\}$.   As $M$ is finite, it is enough to show that $\ovl{C}$ is a finite union of $\ovl{G}$-conjugacy classes.  This follows from Lemma~\ref{lem:finitelinred} as representations in $\ovl{C}$ factor through the finite linearly reductive group $H/N$, so we are done.
\end{proof}

\begin{cor}
\label{cor:normalunipt}
 Suppose $H^0/R_u(H)$ is semisimple, and let $U$ be a maximal unipotent subgroup of $H$.  Let $N$ be a normal unipotent subgroup of $G$.  Let $\sigma\in {\rm Hom}(U, N)$.  Let $C= \{\rho\in {\rm Hom}(H,G)\mid \rho|_U\in G\cdot \sigma\}$.  Then $C$ is a finite union of $G$-conjugacy classes.
\end{cor}

\begin{proof}
 Let $M$ be the subgroup of $H$ generated by all the unipotent elements.  If $\rho\in C$ and $h\in H$ is unipotent then $h$ is $H$-conjugate to an element of $U$ by Proposition~\ref{prop:nonconnmaxunipt}, so $\rho(h)\in N$.  It follows that $M\leq \rho^{-1}(N)$.  Set $\ds M_1= \bigcap_{\rho\in C} \rho^{-1}(N)\unlhd H$ and set $H_1= H/M_1$.  Now $H/M$ is linearly reductive as it consists of semisimple elements; since $M\leq M_1$, $H_1$ is a quotient of $H$, so $H_1$ is also linearly reductive.
 
 Now let $K= \ds \bigcap_{\rho\in C} {\rm ker}(\rho)\unlhd H$, let $H_2= H/K$ and let $\pi_2\colon H\ra H_2$ be the canonical projection.  Then every $\rho\in C$ factors through $H_2$.  If $m\in M_1$ and $\pi_2(m)$ is nontrivial and semisimple then there exists $\rho\in C$ such that $1\neq \rho(m)\in N$; then $\rho(m)$ is both semisimple and unipotent, a contradiction.  This shows that $\pi_2(M_1)$ is unipotent.  Moreover, $H_2/\pi_2(M_1)$ is linearly reductive since it is a quotient of $H_1$, and it is clear that $H_2^0/R_u(H_2)$ is semisimple.  The result now follows from Theorem~\ref{thm:linredKQ} applied to $(G, H_2)$.
\end{proof}

\begin{cor}
\label{cor:abelian}
 Suppose $H/R_u(H)$ is finite and abelian.  Then $(G,H)$ is a K\"ulshammer pair.
\end{cor}

\begin{proof}
 By Corollary~\ref{cor:normaluniptmax}, any maximal unipotent subgroup of $H$ contains $R_u(H)$.  It follows from Proposition~\ref{prop:nonconnmaxunipt}(a) and our hypotheses on $H$ that $H$ has a unique maximal unipotent subgroup $N$ and that every unipotent element of $H$ belongs to $N$.  This implies that $H/N$ is linearly reductive, so the result follows from Theorem~\ref{thm:linredKQ}.
\end{proof}

\begin{cor}
\label{cor:connsolv}
 Suppose $G$ is connected and solvable.  Then $G$ has the K\"ulshammer property.
\end{cor}

\begin{proof}
 Suppose $H^0/R_u(H)$ is semisimple, and let $U$ be a maximal unipotent subgroup of $H$.  If $u\in U$ and $\rho\in {\rm Hom}(H,G)$ then $\rho(u)$ is unipotent, so $\rho(u)$ belongs to $R_u(G)$ as $G$ is connected and solvable.  The result now follows from Corollary~\ref{cor:normalunipt}.
\end{proof}

\begin{cor}
\label{cor:Borel_image}
 Suppose $H^0/R_u(H)$ is semisimple, and let $C$ be as in Corollary~\ref{cor:normalunipt}.  Suppose there exists $\rho\in C$ such that $\rho$ is faithful and $\rho(H)$ is contained in a connected solvable subgroup $M$ of $G$.  Then $(G,H)$ is a K\"ulshammer pair.
\end{cor}

\begin{proof}
 The set $V$ of unipotent elements of $M$ is a normal subgroup of $M$.  Since $\rho$ is faithful, $\rho^{-1}(V)$ is a normal unipotent subgroup of $H$ and $H/\rho^{-1}(V)$ is isomorphic to a subgroup of $M/V$, which is a torus, so $H/\rho^{-1}(V)$ is linearly reductive.  The result follows from Theorem~\ref{thm:linredKQ}.
\end{proof}

Next we prove Theorem~\ref{thm:GcrKQ}.  The proof is similar to that of Theorem~\ref{thm:linredKQ}, but there are some extra complications: for instance, when $p> 0$ we cannot apply Proposition~\ref{prop:subgrpcrit} directly to reduce to the finite case.  In fact, we prove a slight strengthening of Theorem~\ref{thm:GcrKQ}: the latter follows immediately from Theorem~\ref{thm:GcrKQalt} below because any maximal unipotent subgroup of $H^0$ is contained in a maximal unipotent subgroup of $H$ (Proposition~\ref{prop:nonconnmaxunipt}(b)).

\begin{thm}
\label{thm:GcrKQalt}
 Let $G$ be reductive, and suppose $H^0/R_u(H)$ is semisimple.  Let $U$ be a maximal unipotent subgroup of $H^0$ and let $\sigma\in {\rm Hom}(U,G)$.  Then there are only finitely many $G$-conjugacy classes of $G$-cr representations $\rho\in {\rm Hom}(H,G)$ such that $\rho|_U$ is $G$-conjugate to $\sigma$.
\end{thm}

\begin{proof}[Proof of Theorem~\ref{thm:GcrKQ}]
 By Theorem~\ref{thm:char0KQ}, we can assume that $p> 0$.    Let $U$ be a maximal unipotent subgroup of $H^0$.  If $\rho_1, \rho_2\in {\rm Hom}(H^0,G)$ and $\rho_1|_U= \rho_2|_U$ then $\rho_1$ and $\rho_2$ are $C_G(\rho_1(U))$-conjugate by Theorem~\ref{thm:main}.  So fix $\tau\in {\rm Hom}(H^0,G)$ and set $M= \rho(H^0)$.  Let $C= \{\rho\in {\rm Hom}(H,G)_{\rm cr}\mid \rho|_{H^0}= \tau\}$.  By the above discussion, it is enough to show that $C$ is a finite union of $C_G(M)$-conjugacy classes.  We can assume $C$ is nonempty, for otherwise there is nothing to prove.
 
 So fix $\rho\in C$.  As $\rho(H)$ is $G$-cr and $M\unlhd \rho(H)$, $M$ is also $G$-cr \cite[Thm.\ 3.10]{BMR}.  By \cite[Cor.\ 3.17]{BMR}, $C_G(M)$ is $G$-cr, so $C_G(M)$ is reductive.  By Remark~\ref{rem:reduction}, there is a finite subgroup $F$ of $H$ such that $F$ meets every connected component of $H$.  Let $A= \rho(F)C_G(M)$, a reductive subgroup of $N_G(M)$, and let $C'= \{\rho_1|_F\mid \rho_1\in C\}$.  By the same argument as in the proof of Theorem~\ref{thm:char0KQ}, we can regard $C'$ as a subset of ${\rm Hom}(F,A)$.  Note that if $\rho_1, \rho_2\in C$ then $\rho_1$ and $\rho_2$ are $C_G(M)$-conjugate if and only if $\rho_1|_F$ and $\rho_2|_F$ are $C_G(M)$-conjugate.

Let $\rho_1\in C$.  We claim that $K:= \rho_1(F)\leq A$ is $A$-cr.  Now $\rho_1(H)= KM\leq N_G(M)$ is $G$-cr by hypothesis, so $KM/M$ is $N_G(M)/M$-cr by \cite[Prop.\ 6.1(b)]{BMRT}.  Let $\psi\colon A\ra A/C_G(M)^0$ be the canonical projection.  Then $\psi(K)$ is $A/C_G(M)^0$-cr as $A/C_G(M)^0$ is finite.  It follows from \cite[Prop.\ 6.1(c)]{BMRT} that $K$ is $A$-cr, as claimed.

 The claim implies that $C'\subseteq {\rm Hom}(F,A)_{\rm cr}$.  But $F$ is finite, so ${\rm Hom}(F,A)_{\rm cr}$ is a finite union of $A$-conjugacy classes by Theorem~\ref{thm:cr_finite}.  Since $A$ is a finite extension of $C_G(M)$, ${\rm Hom}(F,A)_{\rm cr}$ is a finite union of $C_G(M)$-conjugacy classes.  Hence $C'$ is a finite union of $C_G(M)$-conjugacy classes.  It now follows that $C$ is a finite union of $C_G(M)$-conjugacy classes, as required.
\end{proof}

With the aid of the above results, we can now settle Question~\ref{qn:algKQ} when $G$ is a simple group of rank 1.  (In the special case when $H$ is finite and $p> 0$, this follows already from results described in Section~\ref{sec:intro}, since a simple group of rank 1 is of type $A_1$.)

\begin{cor}
\label{cor:rank1}
 Suppose $G$ is simple and of rank 1.  Then $G$ has the K\"ulshammer property.
\end{cor}

\begin{proof}
 Suppose $H^0/R_u(H)$ is semisimple.  Let $U$ be a maximal unipotent subgroup of $H$ and let $\sigma\in {\rm Hom}(U,G)$.  Let $C= \{\rho\in {\rm Hom}(H,G)\mid \rho|_U= \sigma\}$ and let $C_{\rm cr}= \{\rho\in C\mid \rho\ \mbox{is $G$-cr}\}$.  Then $C_{\rm cr}$ is a finite union of $G$-conjugacy classes by Theorem~\ref{thm:GcrKQ}; in particular, if $\sigma(U)= 1$ then $C= C_{\rm cr}$ and we are done.  So assume $\sigma(U)\neq 1$ and let $B$ be the unique Borel subgroup of $G$ such that $\sigma(U)\leq B$.  Let $\rho\in C$ such that $\rho(H)$ is not $G$-cr.  The image $\rho(H)$ must lie in a Borel subgroup $P$ of $G$, and it is clear that $P= B$.  But $C\cap {\rm Hom}(H,B)$ is a finite union of $B$-conjugacy classes by Corollary~\ref{cor:connsolv}, so the desired result follows.
\end{proof}

We now prove a kind of complementary result to Theorem~\ref{thm:linredKQ} for reductive $G$: rather than having a normal unipotent subgroup and a linearly reductive quotient, we have a normal linearly reductive subgroup and a unipotent quotient.

\begin{prop}
\label{prop:uniptquot}
 Let $G$ be reductive and let $H$ be finite.  Suppose $p> 0$ and $H$ has a normal linearly reductive subgroup $N$ such that $H/N$ is cyclic and of $p$-power order.  Then ${\rm Hom}(H,G)$ is a finite union of $G$-conjugacy classes.  In particular, $(G,H)$ is a K\"ulshammer pair.
\end{prop}

\begin{proof}
 Let $U$ be a Sylow $p$-subgroup of $H$.  Clearly $U$ is a cyclic $p$-group and $H= UN$.  By Lemma~\ref{lem:finitelinred}, ${\rm Hom}(N,G)$ is a finite union of $G$-conjugacy classes.  Fix $\sigma\in {\rm Hom}(N,G)$ and set $C= \{\rho\in {\rm Hom}(H,G)\mid \rho|_N= \sigma\}$; it is enough to show that $C$ is a finite union of $C_G(\sigma(N))$-conjugacy classes (cf.\ the proof of Theorem~\ref{thm:GcrKQalt}).  Now $\sigma(N)$ is $G$-cr as $N$ is linearly reductive, so $N_G(\sigma(N))$ is reductive \cite[Cor.\ 3.16]{BMR} and $N_G(\sigma(N))$ is a finite extension of $C_G(\sigma(N))$ \cite[Lem.\ 6.8]{martin}.  By \cite[Thm.\ 3.3]{guralnick}, $N_G(\sigma(N))$ has only finitely many conjugacy classes of unipotent elements.  This implies that ${\rm Hom}(U,N_G(\sigma(N)))$ is a finite union of $N_G(\sigma(N))$-conjugacy classes, and is therefore a finite union of $C_G(\sigma(N))$-conjugacy classes.  This shows that $C$ is a finite union of $C_G(\sigma(N))$-conjugacy classes, so we are done.
\end{proof}

\begin{rem}
 (a) The result fails if we allow $H/N$ to be a non-cyclic abelian $p$-group \cite[Thm.\ 1.2]{BMR_kuls}, or if we allow $G$ to be non-reductive \cite{cram}.
 
 (b) Conversely, suppose $G$ is reductive, $(G,H)$ is a K\"ulshammer pair and $H$ has a finite cyclic maximal unipotent subgroup $U$.  Then ${\rm Hom}(H,G)$ must be finite, since ${\rm Hom}(U,G)$ is finite by \cite[Thm.\ 3.3]{guralnick}.
\end{rem}

\begin{cor}
 Let $H$ be the dihedral group $D_{2l}$.  Then $(G,H)$ is a K\"ulshammer pair if one of the following holds:\smallskip\\
 (a) $p\neq 2$;\smallskip\\
 (b) $p= 2$, $l$ is odd and $G$ is reductive.
\end{cor}

\begin{proof}
 If $p\neq 2$ then let $q$ be the largest power of $p$ that divides $l$ (taking $q= 1$ if $p= 0$).  The subgroup $C_q$ of $C_l$ is unipotent and normal in $H$, and $H/C_q$ is linearly reductive (having order coprime to $p$), so the result follows from Theorem~\ref{thm:linredKQ}.   If $p= 2$ and $l$ is odd then $C_l$ is a linearly reductive normal subgroup of $H$ and $H/C_l$ is a cyclic 2-group, so the result follows from Proposition~\ref{prop:uniptquot} if $G$ is reductive.
\end{proof}

We finish with some results we will need in Section~\ref{sec:rank2}.

\begin{lem}
\label{lem:allconj}
 Let $G$ be connected and reductive and of semisimple rank at most 2.
Let $M$ be a subgroup of $G$ such that $M$ is not $G$-cr and $M$ is not contained in any Borel subgroup of $G$.  Then there is exactly one proper parabolic subgroup of $G$ that contains $M$.
\end{lem}

\begin{proof}
 By the argument of \cite[Lem.\ 2.12]{BMR}, we can assume that $G$ is semisimple.  As $M$ is not $G$-cr, $M$ is contained in at least one proper parabolic subgroup of $G$.  Let $P_1$, $P_2$ be proper parabolic subgroups of $G$ containing $M$.  Then $P_1\cap P_2$ contains a maximal torus $T$ of $G$, and we can write $P_1= P_{\lambda_1}$ and $P_2= P_{\lambda_2}$ for some $\lambda_1, \lambda_2\in Y(T)$.  By the argument of \cite[Prop.\ 6.7]{martin}, $M$ is contained in $P_{n_1\lambda_1+ n_2\lambda_2}$ for any non-negative integers $n_1$ and $n_2$.  Now $Y(T)$ has rank 2 as a ${\mathbb Z}$-module, so if $\lambda_1$ and $\lambda_2$ are linearly independent over ${\mathbb Z}$ then there exist $n_1,n_2\in {\mathbb N}$ such that $Q:= P_{n_1\lambda_1+ n_2\lambda_2}$ is a Borel subgroup of $G$ (again by the argument of {\em loc.\ cit.}: we just have to choose $n_1$ and $n_2$ in such a way that $\langle n_1\lambda_1+ n_2\lambda_2, \gamma\rangle\neq 0$ for every root $\gamma$ of $G$).  But this contradicts our assumption on $M$, so we must have $a_1\lambda_1= a_2\lambda_2$ for some nonzero integers $a_1$ and $a_2$.
 
 If $a_1$ and $a_2$ have opposite signs then $P_1$ and $P_2$ are opposite to each other, so $M$ is contained in a Levi subgroup $L_1$ of $P_1$.  Since $M$ is not $G$-cr, $M$ is not $L_1$-cr by \cite[Prop.\ 3.2]{serre2}, so $M$ is contained in a proper parabolic subgroup $P_3$ of $L_1$.  Now $L_1$ has semisimple rank 1, so $P_3$ is connected and solvable.  But this implies that $M$ is contained in a Borel subgroup of $G$, a contradiction.  We conclude that $a_1$ and $a_2$ have the same sign, so $P_1= P_2$, as required.
\end{proof}

\begin{prop}
\label{prop:global_to_local}
 Let $G$ be connected and reductive and of semisimple rank at most 2.  Suppose $G$ has the K\"ulshammer property.  Then $P$ has the K\"ulshammer property for every parabolic subgroup $P$ of $G$.
\end{prop}

\begin{proof}
 Let $P$ be a parabolic subgroup of $G$; we show that $P$ has the K\"ulshammer property.  Clearly we can assume $P$ is proper.  By Theorem~\ref{thm:char0KQ} we can assume $p> 0$, and by Proposition~\ref{prop:subgrpcrit} and Remark~\ref{rem:reduction} we can assume $H$ is finite.  If $P$ is a Borel subgroup of $G$ then the result follows from Lemma~\ref{cor:connsolv}, so we can assume that $G$ has semisimple rank 2 and $P$ is a maximal parabolic subgroup of $G$.  Let $U$ be a maximal unipotent subgroup of $H$.  Choose $\sigma\in {\rm Hom}(U,P)$ and let $R\subseteq \{\rho\in {\rm Hom}(H,P)\mid \rho|_U= \sigma\}$.  By induction on $|H|$, we can assume $R$ consists of faithful representations.  Set
 $$ R_1= \{\rho\in R\mid \rho\ \mbox{is $G$-cr}\}, $$
 $$ R_2= \{\rho\in R\mid \rho(H)\ \mbox{is contained in a Borel subgroup of $G$}\}, $$
 $$ R_3= R\backslash (R_1\cup R_2). $$
 If $R_2$ is non-empty then we are done by Corollary~\ref{cor:Borel_image}, so we can assume that $R_2= \emptyset$.  Fix a Levi subgroup $L$ of $P$.  Every representation in $R_1$ is $G$-cr and is therefore $R_u(P)$-conjugate to a subgroup of $L$.  So without loss, we can assume that $R_1\subseteq {\rm Hom}(H,L)$ (note that if $\sigma_1, \sigma_2\in {\rm Hom}(U,L)$ are $R_u(P)$-conjugate then they are equal).  If $M$ is any subgroup of $L$ then $M$ is $G$-cr if and only if $M$ is $L$-cr by \cite[Prop.\ 3.2]{serre2}, so $R_1\subseteq {\rm Hom}(H,L)_{\rm cr}$.  It follows from Theorem~\ref{thm:cr_finite} that $R_1$ is contained in a finite union of $L$-conjugacy classes.
 
 So without loss we can assume $R= R_3$.  By hypothesis, $R_3$ is contained in a finite union of $G$-conjugacy classes.  To finish the proof, it is enough to show that if $\rho_1, \rho_2\in R_3$ are $G$-conjugate then they are $P$-conjugate.  So suppose $\rho_2= g\cdot \rho_1$.  Lemma~\ref{lem:allconj} implies that $P$ is the only proper parabolic subgroup of $G$ that contains $\rho_1(H)$ and $\rho_2(H)$, so $gPg^{-1}= P$.  It follows that $g\in P$, so we are done.
 \end{proof}

\begin{rem}
 In Section~\ref{sec:rank2} we prove a kind of converse to Proposition~\ref{prop:global_to_local} for $G$ a simple group of type $B_2$ when $p= 2$ (see Proposition~\ref{prop:B2_p=2}).  In general, the relationship between Question~\ref{qn:algKQ} for R-parabolic subgroups of $G$ and Question~\ref{qn:algKQ} for $G$ itself is very complicated.  In the latter case, we can deal with $G$-ir representations by Theorem~\ref{thm:GcrKQ}, so it is enough to consider representations $\rho\in {\rm Hom}(H,G)$ with image lying in a proper R-parabolic subgroup $P$ of $G$.  The basic problem with passing between $G$ and $P$ is the following: we can have a subset $R$ of ${\rm Hom}(H,P)$ such that the representations in $R$ are all $G$-conjugate to each other but fall into infinitely many $P$-conjugacy classes.
 
 \begin{ex}
  Suppose ${\rm char}(k)= p> 0$, let $H= C_p\times C_p= \langle h_1, h_2\mid h_1^p= h_2^p= [h_1, h_2]= 1\rangle$ and let $G= {\rm SL_3}(k)$.  Let $B\leq G$ be the parabolic subgroup of upper triangular matrices.  Define $\rho_a\colon H\ra G$ for each $a\in k$ by $\rho_a(h_1)=
\left(
\begin{array}{ccc}
 1 & 0 & 1 \\
 0 & 1 & a \\
 0 & 0 & 1
\end{array}
\right)
$, $\rho_a(h_2)=
\left(
\begin{array}{ccc}
 1 & 0 & 1 \\
 0 & 1 & 1+ a \\
 0 & 0 & 1
\end{array}
\right)
$.  Then $\rho_a= g_a\cdot \rho_0$, where $g_a:=
\left(
\begin{array}{ccc}
 1 & 0 & 0 \\
 a & 1 & 0 \\
 0 & 0 & 1
\end{array}
\right)
$, so the $\rho_a$ are pairwise $G$-conjugate; but it is easily checked that if $a\neq b$ then $\rho_a$ is not $B$-conjugate to $\rho_b$.
 \end{ex}
 
\noindent Above we showed that this phenomenon cannot occur for representations whose images satisfy the hypotheses of Lemma~\ref{lem:allconj}.  For further discussion, see \cite[Sec.\ 3.5]{stewart_PhD}.
\end{rem}

\section{1-cohomology}
\label{sec:H1}

Throughout this section we assume that $G$ is reductive (but not necessarily connected).  Let $P$ be an R-parabolic subgroup of $G$ and let $L$ be an R-Levi subgroup of $P$.  We present an approach to Question~\ref{qn:algKQ} and related problems for $P$ and $L$ using nonabelian 1-cohomology; cf.\ \cite{BMR_kuls}.

We recall some basic material (see \cite[Sec.\ 6]{rich3} for more details).  Let $K$ be an algebraic group, let $V$ be a unipotent group and suppose $K$ acts on $V$ by group automorphisms: that is, suppose $K$ acts on $V$ in the sense of Section~\ref{sec:prelim} and for every $x\in K$, the map $v\mapsto x\cdot v$ is a group automorphism of $V$.  We call a morphism of varieties $\mu:K\rightarrow V$ a \emph{1-cocycle} if $\mu(xy) = \mu(x) (x\cdot\mu(y))$ for all $x,y\in K$ (we refer to this condition as the {\em cocycle equation}).  We call the 1-cocycle given by $\mu(x)= 1$ for all $x\in K$ the {\em trivial 1-cocycle}.  We denote by $Z^1(K, V)$ the set of all 1-cocycles.  For any $\mu\in Z^1(K,V)$ and any $x\in K$, we have $\mu(1)= 1$ and

\begin{equation}
\label{eqn:inverse}
 \mu(x^{-1})= (x^{-1}\cdot \mu(x))^{-1}.
\end{equation}

If $\mu_1, \mu_2\in Z^1(K,V)$ then we define $\mu_1\sim \mu_2$ if there exists $v\in V$ such that $\mu_2(x) = v\mu_1(x)(x \cdot v^{-1})$ for all $x\in K$.  This gives an equivalence relation on $Z^1(K,V)$; we call the equivalence classes {\em 1-cohomology classes} and denote the set of equivalence classes by $H^1(K,V)$.  Given $\mu\in Z^1(K,V)$, we denote by $\ovl{\mu}$ the image of $\mu$ in $H^1(K,V)$.  We define the {\em trivial 1-cohomology class} to be $\ovl{\mu}$, where $\mu$ is the trivial 1-cocycle, and we say that $H^1(K,V)$ is {\em trivial} if the only element of $H^1(K,V)$ is the trivial class.

Consider the special case when $V$ is abelian.  It is easily checked that $Z^1(K,V)$ is an abelian group with respect to pointwise addition of 1-cocycles, and the set $B^1(K,V)$ of 1-coboundaries---that is, the morphisms $\chi_v\colon K\ra V$ given by $\chi_v(h)= v- h\cdot v$ for some fixed $v\in V$---is a subgroup of $Z^1(K,V)$; moreover, we can identify $H^1(K,V)$ with the quotient $Z^1(K,V)/B^1(K,V)$, so $H^1(K,V)$ also has the structure of an abelian group.  (Here we are using additive notation for $V$, $Z^1(K,V)$ and $H^1(K,V)$.)   In this case we refer to ``abelian cohomology".  In general we use the terminology ``non-abelian cohomology'' to signify that $V$ need not be abelian.

Our main source of examples comes from the set-up in the following lemma (cf.\ \cite[Lem.\ 3.2.2]{stewartTAMS}), the proof of which is obtained by straightforward calculation.

\begin{lem}
\label{lem:localdefmn}
 Let $M$ be an algebraic group and let $V\unlhd M$ be unipotent.  Let $\rho\in {\rm Hom}(K,M)$ and let $\mu\colon K\ra V$ be a morphism of varieties.  We let $K$ act on $V$ by $x\cdot v= \rho(x)v\rho(x)^{-1}$.  Define $\rho_\mu\colon K\ra M$ by $\rho_\mu(x)= \mu(x)\rho(x)$.  Then $\rho_\mu$ belongs to ${\rm Hom}(K,M)$ if and only if $\mu$ belongs to $Z^1(K,V)$.  Moreover, if $\mu, \mu'\in Z^1(K,V)$ then $\rho_\mu$ is $V$-conjugate to $\rho_{\mu'}$ if and only if $\ovl{\mu}, \ovl{\mu'}\in H^1(K,V)$ are equal. 
\end{lem}

The next result is  \cite[Lemma 6.2.6]{rich3}.

\begin{lem}
\label{lem:nonab_lin_red}
  Suppose $K$ acts on a unipotent group by group automorphisms and $K$ is linearly reductive. Then $H^1(K, V)$ is trivial.
\end{lem}

\begin{rem}
\label{rem:altpf}
 As an application of this formalism, we give an alternative proof of Lemma~\ref{lem:finitelinred}.  Let $F$ be finite and linearly reductive.  Suppose first that $G$ is reductive.  Then every $\rho\in {\rm Hom}(F,G)$ is $G$-cr, so ${\rm Hom}(F,G)$ is a finite union of $G$-conjugacy classes (Theorem~\ref{thm:cr_finite}).  Now let $G$ be arbitrary, let $V= R_u(G)$ and let $\nu\colon G\ra G/V$ be the canonical projection.  Fix $\tau\in {\rm Hom}(F,G/V)$ and let $C= \{\rho\in {\rm Hom}(F,G)\mid \nu\circ\rho= \tau\}$.  Now ${\rm Hom}(F,G/V)$ is a finite union of $G/V$-conjugacy classes by the reductive case, so it is enough to prove that $C$ is a finite union of $G$-conjugacy classes.
 Fix $\rho\in C$; we let $F$ act on $V$ by the rule $x\cdot v= \rho(x)v\rho(x)^{-1}$.  By Lemma~\ref{lem:nonab_lin_red}, $H^1(F,V)$ is trivial.  The result now follows from Lemma~\ref{lem:localdefmn}.
\end{rem}

We are interested in the following special case of this construction (cf.\ \cite{slodowy}, \cite{BMR_kuls}).  We omit some of the proofs below, as they follow from straightforward diagram-chasing; see \cite[Ch.\ 4]{lond} for further details.  For the rest of the section, we fix an R-parabolic subgroup $P$ of $G$ and an R-Levi subgroup $L$ of $P$.  Set $V= R_u(P)$.  Let $\pi_L\colon P\ra L$ be the canonical projection.  Given $\rho\in {\rm Hom}(K,P)$, set $\rho^L= \pi_L\circ \rho\in {\rm Hom}(K,L)$.  If $R\subseteq {\rm Hom}(K,P)$ then we set $R_L= \{\rho^L\mid \rho\in R\}\subseteq {\rm Hom}(K,L)$.

Fix $\omega\in {\rm Hom}(K,L)$.  Given $R\subseteq {\rm Hom}(K,P)$, set $R_\omega= \{\rho\in R\mid \rho^L= \omega\}$.  We allow $K$ to act on $V$ by the formula
$$ x\cdot v= \omega(x)v\omega(x)^{-1}. $$
We write $Z^1(K,V)_\omega$ and $H^1(K,V)_\omega$ to denote the associated sets of 1-cocycles and 1-cohomology classes.

Now let $\mu\colon K\ra V$ be a morphism of varieties, and define $\rho_\mu\colon K\ra P$ by $\rho_\mu(x)= \mu(x)\omega(x)$.  By Lemma~\ref{lem:localdefmn}, $\rho_\mu$ belongs to ${\rm Hom}(K,P)$ if and only if $\mu$ belongs to $Z^1(K,V)_\omega$, and we get a bijection $z_\omega\colon {\rm Hom}(K,P)_\omega\ra Z^1(K,V)_\omega$ given by $\rho_\mu\mapsto \mu$.  Since $\omega$ is fixed below, we suppress the $\omega$ subscript and write $z$ instead of $z_\omega$; likewise for the maps $h$, $\widetilde{h}$, etc., below.  Below we apply this construction to a second group $K'$ and we write $z'$, $h'$ and $\widetilde{h}'$ for the corresponding maps.  By Lemma~\ref{lem:localdefmn}, $z$ descends to a bijection $h\colon {\rm Hom}(K,P)_\omega/V\ra H^1(K,V)_\omega$.  The group $C_L(\omega(K))$ acts on ${\rm Hom}(K,P)_\omega$ by conjugation and acts on $Z^1(K,V)_\omega$ by the formula $(c\cdot \mu)(x)= c\mu(x)c^{-1}$.  This action descends to give an action of $C_L(\omega(K))$ on $H^1(K,V)_\omega$, and it is straightforward to show that $h$ descends to a bijection $\widetilde{h}\colon {\rm Hom}(K,P)_\omega/C_L(\omega(K))V$ $\ra H^1(K,V)_\omega/C_L(\omega(K))$.

Let $\zeta\colon K'\ra K$ be a homomorphism of algebraic groups.  We allow $K'$ to act on $V$ as above via the representation $\omega':= \omega\circ \zeta\in {\rm Hom}(K',V)$.  We have a map $Z^1(\zeta)\colon Z^1(K,V)_\omega\ra Z^1(K',V)_{\omega'}$ given by $\mu\mapsto \mu\circ \zeta$, and this descends to give maps $H^1(\zeta)\colon H^1(K,V)_\omega\ra H^1(K',V)_{\omega'}$ and $\widetilde{H}^1(\zeta)\colon H^1(K,V)_\omega/C_L(\omega(K))\ra H^1(K',V)_{\omega'}/C_L(\omega'(K'))$.  We have a map $\mathcal{Z}(\zeta)\colon {\rm Hom}(K,P)_\omega\ra {\rm Hom}(K',P)_{\omega'}$ given by $\mathcal{Z}(\zeta)(\rho) = \rho\circ\zeta$.  This descends to give maps $\mathcal{H}(\zeta)\colon {\rm Hom}(K,P)_\omega/V\ra {\rm Hom}(K',P)_{\omega'}/V$ and $\widetilde{\mathcal{H}}(\zeta)\colon {\rm Hom}(K,P)_\omega/C_L(\omega(K))V\ra {\rm Hom}(K',P)_{\omega'}/C_L(\omega'(K'))V$.

The next result, which follows easily from the definitions, summarises the constructions above and says that the various maps involved are all compatible with each other.

\begin{figure}
\par\nobreak
{\small
\setlength{\abovedisplayskip}{6pt}
\setlength{\belowdisplayskip}{\abovedisplayskip}
\setlength{\abovedisplayshortskip}{3pt}
\setlength{\belowdisplayshortskip}{\abovedisplayshortskip}
\begin{align*}
\xymatrix@R=40pt{
\mathrm{Hom}(K, P)_\omega \ar[r]^{z} \ar[d] \ar[dddr]^{\mathcal{Z}(\zeta)} & Z^1(K, V)_\omega \ar[d] \ar[dddr]^{Z^1(\zeta)} & \\
\mathrm{Hom}(K, P)_\omega/V \ar[r]^{h} \ar[d] \ar[dddr]^{\mathcal{H}(\zeta)} & H^1(K, V)_\omega \ar[d] \ar[dddr]^{H^1(\zeta)} & \\
\mathrm{Hom}(K, P)_\omega/C_L(\omega(K))V \ar[r]^{\widetilde{h}} \ar[dddr]^{\widetilde{\mathcal{H}}(\zeta)} & H^1(K, V)_\omega/C_L(\omega(K)) \ar[dddr]^{\widetilde{H^1}(\zeta)} & \\
& \mathrm{Hom}(K', P)_{\omega'} \ar[r]^{\ \ z'} \ar[d] & Z^1(K', V)_{\omega'} \ar[d] \\
& \mathrm{Hom}(K', P)_{\omega'}/V \ar[r]^{\ \ h'} \ar[d] & H^1(K', V)_{\omega'} \ar[d] \\
& \mathrm{Hom}(K', P)_{\omega'}/C_L(\omega'(K'))V \ar[r]^{\ \ \widetilde{h'}} & H^1(K', V)_{\omega'}/C_L(\omega'(K')) \\
}
\end{align*}
}
\caption{Commutative diagram}
\label{fig:cube}
\end{figure}
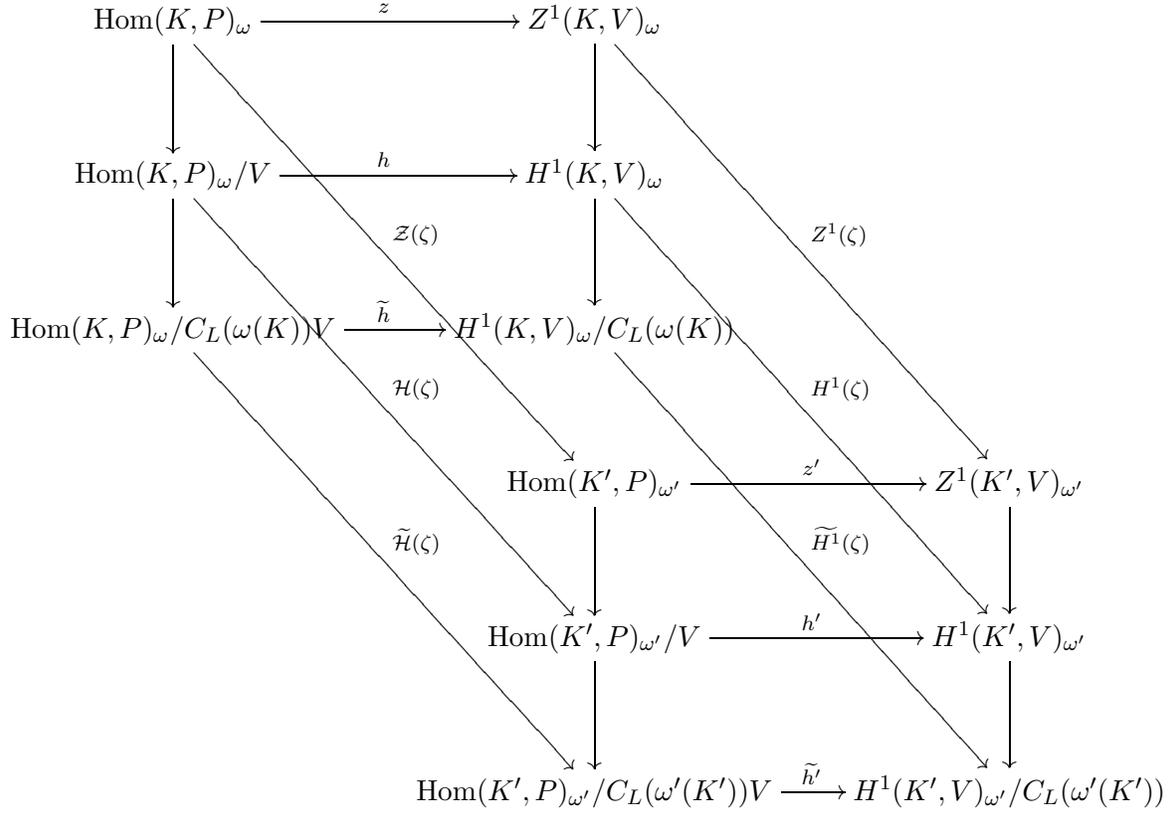

\begin{prop}
\label{prop:cube}
 We have a commutative diagram as in Figure~\ref{fig:cube}, where the vertical maps are the obvious canonical projections.  Moreover, $z$, $h$, $\widetilde{h}$, $z'$, $h'$ and $\widetilde{h'}$ are all bijections.
\end{prop}

The proof of the next result is immediate.

\begin{lem}
\label{lem:conjcrit}
 Let $\omega\in {\rm Hom}(K,L)$ and let $\rho_1, \rho_2\in {\rm Hom}(K, P)_\omega$.  Then $\rho_1$ and $\rho_2$ are $P$-conjugate if and only if they are $C_L(\omega(K))V$-conjugate.
\end{lem}

We can now state our main result of the section.  We have in mind the special case when $K= H$, $K'= U$ is a maximal unipotent subgroup of $H$ and $\zeta$ is the inclusion of $U$ in $H$.

\begin{thm}
\label{main_thm}
 Let $\zeta\colon K'\ra K$ be a homomorphism.  Let $R \subseteq \mathrm{Hom}(K, P)$ and let $S = \{\rho\circ \zeta\mid \rho\in R\}\subseteq {\rm Hom}(K',P)$. Suppose:
\begin{itemize}
\item[(i)] $R_L$ is contained in a finite union of $L$-conjugacy classes;
\item[(ii)] for all $\omega \in {\rm Hom}(K, L)$ such that $R_\omega \neq \emptyset$, the map
\begin{align*}
	\widetilde{H^1}(\zeta):H^1(K, V)_\omega/C_L(\omega(K)) \rightarrow H^1(K', V)_{\omega'}/C_L(\omega'(K'))
\end{align*}
has finite fibres (where $\omega':= \omega\circ \zeta$);
\item[(iii)] $S$ is contained in a finite union of $P$-conjugacy classes.
\end{itemize}
Then $R$ is contained in a finite union of $P$-conjugacy classes.
\end{thm}

\begin{proof}
 Without loss we can assume that (a) all representations in $S$ are $P$-conjugate to each other, and (b) there exists $\omega\in {\rm Hom}(K,L)$ such that $R\subseteq {\rm Hom}(K,P)_\omega$.  Then $S\subseteq {\rm Hom}(K',P)_{\omega'}$, so $\widetilde{h}'(S/C_L(\omega'(K'))V)$ is a single point by Proposition~\ref{prop:cube} and Lemma~\ref{lem:conjcrit} (applied to $K'$).  Let $\widetilde{R}$ be the image of $R$ in ${\rm Hom}(K,P)_\omega/C_L(\omega(K))V$.  Hypothesis (ii) and Proposition~\ref{prop:cube} imply that $\widetilde{h}(\widetilde{R})$ is finite.  It follows from Proposition~\ref{prop:cube} and Lemma~\ref{lem:conjcrit} that $R$ is contained in a finite union of $P$-conjugacy classes, as required.
\end{proof}

\section{Some applications}
\label{sec:appln}

We keep our assumption from Section~\ref{sec:H1} that $G$ is reductive.  We would like to use Theorem~\ref{main_thm} as a tool to answer Question~\ref{qn:algKQ} for an R-parabolic subgroup of $G$; we take $P$ and $V$ to be as above, $K$ to be $H$, $K'$ to be a maximal unipotent subgroup $U$ of $H$ and $\zeta$ to be the inclusion of $U$ in $H$.  The trouble is that in general, hypothesis (ii) of Theorem~\ref{main_thm} is difficult to check.  There are situations when the map $H^1(\zeta)\colon H^1(H,V)_\omega\ra H^1(U,V)_{\omega'}$ is injective.  For instance, Stewart showed that when $H$ is connected and reductive, $H^1(\zeta)$ is injective \cite[Cor.\ 3.4.3 and Thm.\ 3.5.2]{stewartTAMS}, and $H^1(\zeta)$ is also injective if $V$ is abelian (Lemma~\ref{lem:abeliancoh}).  In general, however, the subgroup $C_L(\omega(K))$ from Theorem~\ref{main_thm} is properly contained in $C_L(\omega'(K'))$, so we cannot deduce that the fibres of $\widetilde{H}^1(\zeta)$ are finite without further information.  In particular, we do not know of a cohomological proof of Theorem~\ref{thm:main}, notwithstanding the result of Stewart cited above.  Below we give some situations where the cohomological approach yields fruit.

The first is based on the following observation: if a reductive group $H$ acts by group automorphisms on a unipotent group $V$ and $p= 0$ then $H$ is linearly reductive, so $H^1(H,V)$ is trivial (Lemma~\ref{lem:nonab_lin_red}).  The following result---a variation on Theorem~\ref{thm:linredKQ}---extends this idea.

\begin{prop}
\label{prop:linred}
 Let $P$, $L$, $V$ and $\pi_L$ be as in Section~\ref{sec:H1}, and let $U$ be a maximal unipotent subgroup of $H$.  Let $\omega\in {\rm Hom}(H,L)$, let $\rho_1,\rho_2\in {\rm Hom}(H,P)_\omega$ such that $\rho_1|_U= \rho_2|_U$, and set $N= {\rm ker}(\omega)$.  Suppose $H/N$ is linearly reductive.  Then $\rho_2$ is $C_V(\rho_1(U))$-conjugate to $\rho_1$.
\end{prop}

\begin{proof}
 We allow $H$ to act on $V$ by the rule $h\cdot v= \rho_1(h)v\rho_1(h)^{-1}$.  Define $\mu(h)= \rho_2(h)\rho_1(h)^{-1}$.  Then $\mu\in Z^1(H,V)$ by Lemma~\ref{lem:localdefmn}.  For $i= 1,2$, let $\psi_i\colon H\ra H/{\rm ker}(\rho_i)$ be the canonical projection.  Then $\psi_i(N)$ is isomorphic to a subgroup of $V$, so $\psi_i(N)$ is unipotent, and $\psi_i(N)\unlhd H/{\rm ker}(\rho_i)$.  By Corollary~\ref{cor:normaluniptmax} and Proposition~\ref{prop:epi}, $\psi_i(N)\leq\psi_i(U)$, so $N\leq U\,{\rm ker}(\rho_i)$.  It follows that $\rho_1(n)= \rho_2(n)$ for any $n\in N$, and this implies that $\mu(n)= 1$ for all $n\in N$.
 
 For any $h\in H$ and any $n\in N$,
 $$ 1= \mu(hnh^{-1})= \mu(h)(hn\cdot \mu(h^{-1}))= \mu(h)(hnh^{-1}\cdot \mu(h)^{-1}) $$
 by the cocycle equation and (\ref{eqn:inverse}).  It follows that $\mu(H)\subseteq \widetilde{V}:= C_{\rho_1(N)}(V)$.  We may, therefore, regard $\mu$ as a 1-cocycle for the $H$-module $\widetilde{V}$.  In fact, since $N$ acts trivially on $\widetilde{V}$ and $\mu|_N$ is trivial, we may regard $\mu$ as an element of $Z^1(H/N,\widetilde{V})$.  As $H/N$ is linearly reductive, $H^1(H/N,\widetilde{V})$ is trivial.  It follows that $\ovl{\mu}$ is the trivial element of $H^1(H,V)$.  Lemma~\ref{lem:localdefmn} now implies that $\rho_2= g\cdot \rho_1$ for some $g\in V$.  As $\rho_1$ and $\rho_2$ agree on $U$, $g$ belongs to $C_V(\rho_1(U))$, so we are done.
\end{proof}

\begin{rem}
 If $H$ is a nontrivial torus then $H$ is linearly reductive.  In fact, $H^1(H,V)$ is trivial for any action of $H$ by group automorphisms on a unipotent group $V$ (Lemma~\ref{lem:nonab_lin_red}), so the fibres of $\widetilde{H}^1(\zeta)$ in Theorem~\ref{main_thm} are automatically trivial for $K= H$.  But $(G,H)$ is usually not a K\"ulshammer pair for the reasons described in Section~\ref{sec:intro}; this is not detected by the 1-cohomology. 
\end{rem} 

Our next application concerns the case when $V$ is abelian.   We need a preliminary lemma.

\begin{lem}
\label{lem:abeliancoh}
 Let $H$ act by group automorphisms on an abelian unipotent group $V$ and let $U$ be a maximal unipotent subgroup of $H$.  Then the map $H^1(\zeta)\colon H^1(H,V)\ra H^1(U,V)$ is injective.
\end{lem}

\begin{proof}
 One checks easily that $H^1(\zeta)$ is a homomorphism of abelian groups.  It is therefore enough to prove that if $\mu\in Z^1(H,V)$ and $\mu|_U$ is a 1-coboundary then $\mu$ is a 1-coboundary.
 
 Suppose first that $p= 0$.  We have a Levi factorisation $H= M\ltimes R_u(H)$ \cite[Thm.\ 7.1]{mostow}.  Let $\mu\in Z^1(H,V)$ such that $\mu|_U$ is a 1-coboundary.  By adding a 1-coboundary to $\mu$ if necessary, we can assume that $\mu|_U= 0$.  For any $u\in R_u(H)$ and any $m\in M$,
 $$ 0= \mu(mum^{-1})= \mu(m)- mum^{-1}\cdot \mu(m) $$
 by (\ref{eqn:inverse}).  Hence $R_u(H)$ centralises $\mu(m)$.  It follows that we may regard $\mu|_M$ as an element of $Z^1(M, V_1)$, where $V_1$ is the fixed point subgroup $V^{R_u(H)}$.  Now $M$ is linearly reductive, so $H^1(M, V_1)= 0$ by Lemma~\ref{lem:nonab_lin_red}.  We deduce that there exists $v\in V_1$ such that $\mu(m)= v- m\cdot v$ for all $m\in M$.  As $v$ centralises $R_u(H)$ and $\mu|_{R_u(H)}= 0$, it follows that $\mu$ is the 1-coboundary $\chi_v$, so we are done.
 
 Now suppose $p> 0$.  Suppose first that $H$ is finite, and let $U$ be a Sylow $p$-subgroup of $H$.  The proof in this case proceeds by a standard averaging argument (cf.\ \cite[III.10]{brown}).  Let $\mu\in Z^1(H,V)$ such that $\mu|_U$ is a 1-coboundary; as in the previous paragraph, we can assume that $\mu|_U= 0$.  Let $r= |H:U|$.  Now $V$ is unipotent, so $V$ has exponent $q$ for some power $q$ of $p$.  As $r$ is coprime to $p$, there exists $s\in {\mathbb N}$ such that $srv= v$ for all $v\in V$.  Let $t_1,\ldots, t_r$ be a set of representatives for the coset space $H/U$.  Set $\ds v= s\sum_{i= 1}^r \mu(t_i)$.  If $h\in H$ then there is a permutation $\sigma$ of $\{1,\ldots, r\}$ such that for any $i$, $ht_i= t_{\sigma(i)}u_i$ for some $u_i\in U$.  As $\mu$ is constant on each coset in $H/U$ and $U$ is abelian, we deduce that $\ds v= s\sum_{i= 1}^r \mu(ht_i)$.  But $\ds s\sum_{i= 1}^r \mu(ht_i)= s\sum_{i=1}^r (\mu(h)+ h\cdot \mu(t_i))= sr\mu(h)+ h\cdot v$, so $\mu(h)= v- h\cdot v$.  Hence $\mu$ is a 1-coboundary.

 Now let $H$ be arbitrary.  By Remark~\ref{rem:reduction}, there exist finite subgroups $H_1\leq H_2\leq \cdots$ of $H$ such that $\ds \bigcup_{m= 1}^\infty H_m$ is dense in $H$.  By Proposition~\ref{prop:nonconnmaxunipt}, we can choose a maximal unipotent subgroup $U_m$ of $H_m$ in such a way that $U_1\leq U_2\leq\cdots$.  Then the closure of $\ds \bigcup_{m\in {\mathbb N}} U_m$ is contained in some maximal unipotent subgroup $U$ of $H$ (cf.\ the proof of Proposition~\ref{prop:subgrpcrit}).
 
 Let $\mu\in Z^1(H,V)$ such that $\mu|_U$ is a 1-coboundary.  Then for each $m$, $\mu|_{U_m}$ is a 1-coboundary, so $\mu|_{H_m}$ is a 1-coboundary by the finite case.  Set
 $$ C_m= \{v\in V\mid \mu(h)= v- h\cdot v \ \mbox{for all $h\in H_m$}\}. $$
Then $C_1, C_2,\ldots$ is a descending sequence of nonempty closed subsets of $V$.  By the descending chain condition, this sequence must eventually become constant, so there exists $v\in V$ such that $\mu(h)= v- h\cdot v$ for all $m\in {\mathbb N}$ and all $h\in H_m$.  But $\ds \bigcup_{m= 1}^\infty H_m$ is dense in $H$ and $\mu$ is a morphism, so $\mu(h)= v- h\cdot v$ for all $h\in H$.  This completes the proof.
\end{proof}

\begin{rem}
 The conclusion of Lemma~\ref{lem:abeliancoh} can fail if we don't assume that $V$ is abelian: see \cite[Sec.\ 3]{BMR_kuls}.
\end{rem}

\begin{prop}
\label{prop:Levicont}
 Let $P$ be an R-parabolic subgroup of $G$ and suppose $V:= R_u(P)$ is abelian.  Let $H\leq P$ and let $U$ be a maximal unipotent subgroup of $H$.  If $U$ is contained in an R-Levi subgroup $L$ of $P$ then $H$ is contained in a $C_V(U)$-conjugate of $L$.  In particular, $H$ is also contained in an R-Levi subgroup of $P$.
\end{prop}

\begin{proof}
 Let $\rho\colon H\ra P$ be the inclusion of $H$ in $P$.  Suppose $U\leq L$, where $L$ is an R-Levi subgroup of $P$.  Set $\omega= \rho^L\in {\rm Hom}(H,L)$.  By Lemma~\ref{lem:localdefmn}, there exists $\mu\in Z^1(H,V)_\omega$ such that $\rho(h)= \mu(h)\omega(h)$ for all $h\in H$.  Now $\mu(u)= 1$ for all $u\in U$, and by Lemma~\ref{lem:abeliancoh}, $H^1(\zeta)$ is injective.  It follows that $\ovl{\mu}$ is trivial in $H^1(H,V)$, so $\rho= g\cdot \sigma$ for some $g\in V$ by Lemma~\ref{lem:localdefmn}.  As $\rho$ and $\sigma$ agree on $U$, $g$ must belong to $C_V(U)$, so we are done.
\end{proof}

\begin{rem}
 If $H$ is connected and reductive then Proposition~\ref{prop:Levicont} holds without the assumption that $V$ is abelian: see \cite[Cor.\ 3.6.2]{stewartTAMS}.
\end{rem}

We finish with a result very similar to Proposition~\ref{prop:Levicont}, but formulated for representations instead of subgroups.

\begin{lem}
\label{lem:abelianpar}
 Let $G$ be reductive and let $P$ be an R-parabolic subgroup of $G$ such that $V:= R_u(P)$ is abelian.  Fix an R-Levi subgroup $L$ of $P$ and let $\omega\in {\rm Hom}(H,L)$.  Let $\rho_1, \rho_2\in {\rm Hom}(H,P)_\omega$.  If $\rho_1|_U$ and $\rho_2|_U$ are $C_L(\omega(H))V$-conjugate then $\rho_1$ and $\rho_2$ are $C_L(\omega(H))V$-conjugate.
\end{lem}

\begin{proof}
 Suppose $\rho_1|_U= g\cdot \rho_2|_U$ for some $g\in C_L(\omega(H))V$.  Replacing $\rho_2$ with $g\cdot \rho_2$, we can assume that $\rho_1|_U= \rho_2|_U$ (note that $g\cdot \rho_2$ belongs to ${\rm Hom}(H,P)_\omega$ by our assumption on $g$).  By Lemma~\ref{lem:localdefmn}, we can write $\rho_1(h)= \mu_1(h)\omega(h)$ and $\rho_2(h)= \mu_2(h)\omega(h)$ for some $\mu_1, \mu_2\in Z^1(H,V)_\omega$.  Since $\mu_1|_U= \mu_2|_U$ and $V$ is abelian, Lemma~\ref{lem:abeliancoh} implies that the elements $\ovl{\mu_1}$ and $\ovl{\mu_2}$ of $H^1(H,V)_\omega$ are equal.  Hence $\rho_1$ and $\rho_2$ are $V$-conjugate by Lemma~\ref{lem:localdefmn}, so we are done.
\end{proof}

\section{Groups of semisimple rank 2}
\label{sec:rank2}

In \cite{BMR_kuls}, Bate, Martin, and R\"ohrle produce a finite subgroup $H$ of a simple group $G$ of type $G_2$ with $p= 2$ such that $(G, H)$ is not a K\"ulshammer pair.  We show that this is the smallest example possible, in the sense that there are no other such examples for any semisimple $G\neq G_2$ of rank 1 or 2.  Our proof uses the cohomological formalism from Section~\ref{sec:H1}, as well as various results from Sections~\ref{sec:mainproof} and \ref{sec:appln}.

\begin{thm}
\label{thm:lowrank}
 Suppose $H^0/R_u(H)$ is semisimple.  Let $G$ be a semisimple group of rank at most 2, and assume that if $p= 2$ or 3 then $G$ is not of type $G_2$.  Then $(G,H)$ is a K\"ulshammer pair.
\end{thm}

We need some preliminary work.  The key case to deal with is that of type $B_2$ in characteristic 2.  Until the end of the proof of Proposition~\ref{prop:B2_localtoglobal}, we assume that $p= 2$, $G$ is a simply connected simple group of type $B_2$ (so $G= {\rm Spin}_5$) and $H$ is finite.  Fix a maximal unipotent subgroup $U$ of $H$.  We fix a maximal torus $T$ of $G$ and a Borel subgroup $B$ of $G$ such that $T\leq B$.  We label the roots of $B$ with respect to $T$ as $\alpha, \beta, \alpha+ \beta$ and $2\alpha+ \beta$, where $\alpha$ is short and $\beta$ is long.  The hypothesis of simple connectedness ensures that the canonical epimorphism from ${\rm SL}_2(k)$ to $G_\beta$ is an isomorphism.  The commutation relations of the root groups are given in \cite[33.4]{hum}.  If $p= 2$ then the root groups $U_\gamma$ and $U_\delta$ commute with each other for any roots $\gamma, \delta$ of $B$ except for when $\{\gamma, \delta\}= \{\alpha, \beta\}$, whereas for any $1\neq u_\alpha\in U_\alpha$ and $1\neq u_\beta\in U_\beta$, we have

\begin{equation}
\label{eqn:noncomm}
 [u_\alpha, u_\beta]= u_{\alpha+ \beta} u_{2\alpha+ \beta}
\end{equation}

\noindent for some $1\neq u_{\alpha+ \beta}\in U_{\alpha+ \beta}$ and some $1\neq u_{2\alpha+ \beta}\in U_{2\alpha+ \beta}$.  Moreover, $[G_{\alpha}, G_{\alpha+ \beta}]= [G_\beta, G_{2\alpha+ \beta}]= 1$.

There are two $G$-conjugacy classes of maximal parabolic subgroups of $G$, represented by $P_\alpha:= \langle B\cup U_{-\alpha}\rangle$ and $P_\beta:= \langle B\cup U_{-\beta}\rangle$.  Below we need only consider $P_\beta$ (cf.\ the proof of Proposition~\ref{prop:B2_p=2}).  Set $P= P_\beta$.  Define $L= \langle T\cup G_\beta\rangle$, a Levi subgroup of $P$.

We have $R_u(P)= U_\alpha U_{\alpha+ \beta} U_{2\alpha+ \beta}$ and $R_u(B)= U_\beta U_\alpha U_{\alpha+ \beta} U_{2\alpha+ \beta}= U_\beta R_u(P)$.  Note that $R_u(P)$ is abelian.  The subgroup $U_{2\alpha+ \beta}$ is normal in $P$; we set $Q= P/U_{2\alpha+ \beta}$ and denote by $\xi$ the canonical projection from $P$ to $Q$.  Set $V= \xi(R_u(P))$.

\begin{lem}
\label{lem:almost_par}
 $(Q,H)$ is a K\"ulshammer pair.
\end{lem}

\begin{proof}
 The centre $Z(L)$ of $L$ is the image of the coroot $(2\alpha+ \beta)^\vee$.  Define $f\colon (k^*)^2\ra T$ by $f(x,y)= \beta^\vee(x) (2\alpha+ \beta)^\vee(y)$ and $h\colon T\ra (k^*)^2$ by $h(t)= (\chi_1(t), \chi_2(t))$, where $\chi_1$ and $\chi_2$ are the fundamental dominant weights given by $\chi_1= \alpha+ \frac{1}{2}\beta$ and $\chi_2= \alpha+ \beta$.  A short calculation shows that $h\circ f$ is an isomorphism, so $f$ is an isomorphism onto $T$.  It follows that $L\iso Z(L)\times [L,L]\iso k^*\times {\rm SL}_2(k)$.
 
 The multiplication map $k^*\times {\rm SL}_2(k)\ra {\rm GL}(2,k)$ gives an isomorphism of abstract groups from $L$ to ${\rm GL}(2,k)$.  It is easily checked that this extends to an isomorphism of abstract groups from $L\ltimes V$ to ${\rm GL}_2(k)\ltimes V$, where ${\rm GL}_2(k)$ acts on $V= k^2$ by the natural representation (note that $\langle \alpha, -(2\alpha+ \beta)^\vee\rangle= \langle \alpha+ \beta, -(2\alpha+ \beta)^\vee\rangle= 1$, so $Z(L)$ acts on $V$ with weight 1).  Now ${\rm GL}_2(k)\ltimes V$ is isomorphic to a maximal parabolic subgroup of ${\rm SL}_3(k)$; as ${\rm SL}_3(k)$ has the K\"ulshammer property (see Section~\ref{sec:intro}), it follows from Lemma~\ref{prop:global_to_local} that ${\rm GL}_2(k)\ltimes V$ has the K\"ulshammer property.  We conclude from Remark~\ref{rem:abstractiso} that $(Q,H)$ is a K\"ulshammer pair, as required.
\end{proof}.

We adopt the following notation: given a set $S$ and a function $f\colon S\ra P$, we write
$$ f(s)= f_L(s) f_\alpha(s) f_{\alpha+ \beta}(s) f_{2\alpha+ \beta}(s) $$
for $s\in S$, where $f_L\colon S\ra L$ is a function and $f_\gamma\colon S\ra U_\gamma$ is a function for all $\gamma\in \{\alpha, \alpha+ \beta, 2\alpha+ \beta\}$.  If $f_L(S)\subseteq U_\beta$ then we write $f_\beta(s)$ for $f_L(s)$.

Define
$$ {\rm Hom}(U, P)'= \{\sigma\in {\rm Hom}(U, P)\mid \mbox{$\sigma$ is faithful, $\sigma(U)\leq U_\beta R_u(P)$ and $\sigma_\beta(U)\neq 1$}\}, $$
$$ {\rm Hom}(H, P)'= \{\rho\in {\rm Hom}(H, P)\mid \mbox{$\rho$ is faithful, $\rho|_U\in {\rm Hom}(U, P)'$ and $\rho^L$ is $L$-ir}\} $$
and
$$ {\rm Hom}(H,L)'= \{\omega\in {\rm Hom}(H, L)\mid \mbox{$\omega$ is faithful, $1\neq \omega(U)\leq U_\beta$ and $\omega$ is $L$-ir}\}. $$
Note that if $\rho\in {\rm Hom}(H,P)$ and $\rho^L$ belongs to ${\rm Hom}(H,L)'$ then $\rho$ belongs to ${\rm Hom}(H, P)'$, but the converse need not hold because $\rho^L$ need not be faithful.  Given $\sigma\in {\rm Hom}(U, P)'$, we call $\sigma$ {\em regular} if $\sigma_\alpha(u)\neq 1$ for some $u\in U$.  Otherwise we call $\sigma$ {\em singular}.  If $\rho\in {\rm Hom}(H, P)'$ then we call $\rho$ {\em regular} if $\rho|_U$ is regular, and {\em singular} otherwise.

\begin{lem}
\label{lem:regular}
 (a) Let $v= v_\beta v_\alpha v_{\alpha+ \beta} v_{2\alpha+ \beta}\in R_u(B)$ with $v_\gamma\in U_\gamma$ for each $\gamma\in \{\beta, \alpha, \alpha+ \beta,  2\alpha+ \beta\}$.  Then $v$ is a regular unipotent element of $G$ if and only $v_\beta\neq 1$ and $v_\alpha\neq 1$.\smallskip\\
 (b) Let $\sigma\in {\rm Hom}(U, P)'$.  Then $\sigma$ is regular if and only if $\sigma(u_1)$ is regular for some $u_1\in U$.\smallskip\\
 (c) Let $\rho\in {\rm Hom}(H, P)'$.  Then $\rho$ is regular if and only if $\rho(u_1)$ is regular for some $u_1\in U$.\smallskip\\
 (d) Let $\rho\in {\rm Hom}(H, P)'$.  Suppose there exists $u_2\in U$ such that $\rho_\beta(u_2)= \rho_\alpha(u_2)= 1$ and $\rho_{\alpha+ \beta}(u_2)\neq 1$.  Then there exists $u_3\in U$ such that $\rho_\beta(u_3)= 1$ and $\rho_\alpha(u_3)\neq 1$.  In particular, $\rho$ is regular.\smallskip\\
 (e) Let $\rho\in {\rm Hom}(H, P)'$.  Then $\rho$ is regular if and only if there exists $u_3\in U$ such that $\rho_\beta(u_3)= 1$ and $\rho_\alpha(u_3)\neq 1$.
\end{lem}

\begin{proof}
 Part (a) is standard (see, e.g., \cite[Ch.\ 4]{hum_conj}), and parts (b) and (c) are straightforward.\smallskip\\
  (d) Since $\rho^L$ is $L$-irreducible by hypothesis, $H$ acts irreducibly on $V$ via $\xi\circ \rho^L$, so there exists $h\in H$ such that $\rho_\alpha(hu_2h^{-1})\neq 1$.  Set $u_3= hu_2h^{-1}$.  Then $u_3$ belongs to $\rho^{-1}(R_u(P_\alpha))$, which is a normal unipotent subgroup of $H$ as $\rho$ is faithful, so $u_3\in U$ by Corollary~\ref{cor:normaluniptmax}.  Clearly $u_3$ has the desired properties.\smallskip\\ 
 (e) If $\rho_\alpha(u_3)\neq 1$ for some $u_3\in U$ then $\rho$ is regular by definition.  Conversely, suppose $\rho$ is regular.  By parts (a) and (c), there exists $u_1\in U$ such that $\rho_\beta(u_1)\neq 1$ and $\rho_\alpha(u_1)\neq 1$.  Eqn.\ (\ref{eqn:noncomm}) implies that $u_2:= u_1^2$ satisfies the hypotheses of (d).  Part (d) applied to $u_2$ yields $u_3$ with the desired properties.
\end{proof}

\begin{lem}
\label{lem:faithful}
 Let $\omega\in {\rm Hom}(H,L)'$.  Then ${\rm Hom}(H, P)_\omega$ is a union of at most two $Z(L)R_u(P)$-conjugacy classes.
\end{lem}

\begin{proof}
 Let $F= \omega(H)\leq L\iso {\rm GL}_2(k)$.  Since $F$ is ${\rm GL}_2(k)$-ir, the inclusion of $F$ in ${\rm GL}_2(k)$ is an irreducible representation of $F$; in particular, $F$ is not cyclic.  Standard representation-theoretic results imply that $F\leq {\rm GL}_2(q)$ for some power $q$ of 2.  Let $\psi\colon {\rm GL}_2(q)\ra PGL_2(q)\iso {\rm SL}_2(q)$ be the canonical projection: then $\psi(F)\leq {\rm SL}_2(q)$.  Let $q_1$ be the smallest power of 2 such that $\psi(F)$ is isomorphic to a subgroup of ${\rm SL}_2(q_1)$.  By \cite[Cor.\ 2.2]{king}, any absolutely irreducible maximal proper subgroup of ${\rm SL}_2(q_1)$ is isomorphic to ${\rm SL}_2(q_0)$ for some power $q_0$ of 2 with $q_0\leq q_1$ or to a dihedral group $D_{2r}$ of order $2r$ for some odd $r$.  Minimality of $q_1$ implies that $\psi(F)\iso {\rm SL}_2(q_1)$ or $\psi(F)$ is a dihedral group $D_{2s}$ for some odd $s$.
 
 Suppose $\psi(F)\iso {\rm SL}_2(q_1)$.  Let $A$ be a Sylow $2$-subgroup of $H$; we can choose $A$ so that $\omega(\psi(A))$ is the group of upper unitriangular matrices in ${\rm SL}_2(q_1)$.  Choose $D\leq H$ such that $\omega(\psi(D))$ is the group of diagonal matrices in ${\rm SL}_2(q_1)$.  Then $D$ is abelian and consists of semisimple elements, $D$ normalises $A$ and $D$ acts transitively on the set of nontrivial elements of $A$.  There is no harm in replacing $\omega$ with an $L$-conjugate of $\omega$, so we may assume that $\psi(D)\leq T$ and $\psi(A)\leq U_\beta$.
 
 Recall from Section~\ref{sec:H1} that $R_u(P)$-conjugacy classes of elements of ${\rm Hom}(H,P)_\omega$ correspond bijectively to elements of $H^1(H,R_u(P))_\omega$.  Set $Z^1(H,R_u(P))_\omega^0= \{\mu\in Z^1(H,R_u(P))_\omega\mid \mu|_D\ \mbox{is trivial}\}$.  We claim that $Z^1(H,R_u(P))_\omega^0$ surjects onto $H^1(H,R_u(P))_\omega$.  To see this, let $\mu\in Z^1(H,R_u(P))_\omega$.  Since $H^1(D,R_u(P))_\omega$ vanishes (Lemma~\ref{lem:nonab_lin_red}), the restriction $\mu|_D$ is a 1-coboundary, so there exists $v\in R_u(P)$ such that the 1-cocycle $\mu'$ given by $\mu'(h)= \mu(h)- v+ h\cdot v$ belongs to $Z^1(H,R_u(P))_\omega^0$.  We have $\ovl{\mu'}= \ovl{\mu}$, so the claim is proved.
 
 So let $\mu\in Z^1(H, R_u(P))_\omega^0$.  It follows from the cocycle equation that for any $x\in D$,
 \begin{equation}
 \label{eqn:determined}
  \mu(xa_0x^{-1})= x\cdot \mu(a_0).
 \end{equation}
Fix $1\neq a_0\in A$.  The cocycle equation together with (\ref{eqn:noncomm}) applied to $a_0^2= 1$ implies that $\mu(a_0)$ belongs to $U_{\alpha+ \beta}U_{2\alpha+ \beta}$.  It follows that $A$ acts trivially on $\mu(A)$, so $\mu$ is a homomorphism from $A$ to $U_{\alpha+ \beta}U_{2\alpha+ \beta}$.  Eqn.\ (\ref{eqn:determined}) implies that $\mu_{2\alpha+ \beta}(a)= 1$ for all $a\in A$ as $D$ centralises $U_{2\alpha+ \beta}$; in particular, $\mu(a_0)\in U_{\alpha+ \beta}$.   Since any two nontrivial elements of $A$ are $D$-conjugate, Eqn.\ (\ref{eqn:determined}) implies that $\mu|_A$ is completely determined by $\mu(a_0)$.  If $\mu(a_0)= 1$ then $\mu|_A$ is trivial, so $\ovl{\mu}\in H^1(H,R_u(P))_\omega$ is trivial by Lemma~\ref{lem:abeliancoh}.  Finally, the conjugation action of $Z(L)^0$ is transitive on the set of nontrivial elements of $U_{\alpha+ \beta}$.  Putting these facts together, we deduce that if $H^1(H,R_u(P))_\omega$ is nontrivial then the conjugation action of $Z(L)^0$ is transitive on the set of nontrivial elements of $H^1(H,R_u(P))_\omega$.  The result follows.
 
 Now suppose $\psi(F)= D_{2s}$ where $s> 1$ is odd.  Choose $D\leq H$ such that $\omega(\psi(D))$ is the cyclic subgroup $C_s$ of $D_{2s}$ and let $a\in H$ be an involution; then $a$ normalises $D$.  Set $Z^1(H,R_u(P))_\omega^0= \{\mu\in Z^1(H,R_u(P))_\omega\mid \mu|_D\ \mbox{is trivial}\}$.  For any $\mu\in Z^1(H,R_u(P))_\omega^0$ and any $x\in D$,
 $$ \mu(a)= \mu(ax)= \mu(x^{-1}a)= x^{-1}\cdot \mu(a). $$
 It follows that $\mu(a)\in U_{2\alpha+ \beta}$.  But the conjugation action of $Z(L)^0$ is transitive on the set of nontrivial elements of $U_{2\alpha+ \beta}$, so the desired result follows as in the previous case. 
\end{proof}

\begin{prop}
\label{prop:B2}
 $(P,H)$ is a K\"ulshammer pair.
\end{prop}

\begin{proof}
 Fix $\sigma\in {\rm Hom}(U, P)$.  Let $C\subseteq \{\rho\in {\rm Hom}(H,P)\mid \rho|_U= \sigma\}$.  We show that $C$ is contained in a finite union of $P$-conjugacy classes.  We are free to replace $\sigma$ with a $P$-conjugate of $\sigma$, so we can assume that $\sigma(U)\leq U_\beta$.  By induction on $|H|$, we can assume that $C$ consists of faithful representations; in particular, we can assume that $\sigma$ is faithful.  If $\sigma(U)\leq R_u(P)$ then $U= \sigma^{-1}(R_u(P))$ is a normal unipotent subgroup of $H$, and the result follows from Theorem~\ref{thm:linredKQ} and Remark~\ref{rem:Nclassfn}.  Hence we can assume that $\sigma_\beta(U)\neq 1$.  If $\rho\in C$ and $\rho^L$ is not $L$-ir then $\rho^L(H)$ is contained in a Borel subgroup of $L$, so $\rho(H)$ is contained in a Borel subgroup of $P$ and the result follows from Corollary~\ref{cor:Borel_image}.  Hence we can assume that $C\subseteq {\rm Hom}(H,P)'$.
 
 As $(Q, H)$ is a K\"ulshammer pair (Lemma~\ref{lem:almost_par}), we can assume that the representations $\xi\circ \rho$ are pairwise $Q$-conjugate as $\rho$ ranges over the elements of $C$.  Fix $\rho_0\in C$ and set $g_{\rho_0}= 1$.  Set $\omega= \rho_0^L$; then $\omega|_U= \sigma^L$.  For each $\rho\in C$ with $\rho\neq \rho_0$, choose $g_\rho\in P$ such that
 \begin{equation}
 \label{eqn:Qproj}
  \xi\circ (g_\rho\cdot \rho)= \xi\circ \rho_0,
 \end{equation}
 and set $C_1= \{g_\rho\cdot \rho\mid \rho\in C\}\subseteq {\rm Hom}(H,P)'_\omega$.  Then
 \begin{equation}
 \label{eqn:tauproj}
   \tau^L|_U= \sigma^L \ \mbox{for all $\tau\in C_1$}.
 \end{equation}
So the representations $\tau|_U$ for $\tau\in C_1$ need not {\em a priori} all be equal to $\sigma$, but $\tau(u)$ has the form $\tau_\beta(u)\tau_\alpha(u)\tau_{\alpha+ \beta}(u)\tau_{2\alpha+ \beta}(u)$, where $\tau_\beta(u)= \sigma_\beta(u)$, $\tau_\alpha(u)= \sigma_\alpha(u)$ and $\tau_{\alpha+ \beta}(u)= \sigma_{\alpha+ \beta}(u)$.
 
 Fix $\rho\in C$.  We claim that 
 \begin{equation}
 \label{eqn:Ru_conj}
  (g_\rho\cdot \rho)|_U= m\cdot \sigma \ \mbox{for some $m\in Z(L)R_u(P)$}.
 \end{equation}
To establish this, we argue as follows.  We can write $g_\rho= lv_\alpha v_{\alpha+ \beta} v_{2\alpha+ \beta}$ for some $l\in L$ and some $v_\gamma\in U_\gamma$ ($\gamma= \alpha, \alpha+ \beta, \alpha+ 2\beta$).  Since $\sigma(U)\leq U_\beta$, it follows from (\ref{eqn:tauproj}) that $l= zv_\beta$ for some $z\in Z(L)$ and some $v_\beta\in U_\beta$, so we have
$$ g_\rho= zv_\beta v_\alpha v_{\alpha+ \beta} v_{2\alpha+ \beta}. $$
If $v_\beta$ commutes with $\sigma(U)$ then a short calculation using the commutation relations for the root groups shows that $(g_\rho\cdot \rho)|_U= m\cdot \sigma$, where $m:= zv_\alpha$, and the claim is proved.  In particular, this is the case if $\sigma$ is singular or $v_\beta= 1$.  So let us suppose that $\sigma$ is regular and $v_\beta\neq 1$.  We show this leads to a contradiction.

By Lemma~\ref{lem:regular}(a) and (b), there exists $u_1\in U$ such that $\sigma_\beta(u_1)\neq 1$ and $\sigma_\alpha(u_1)\neq 1$.  Now $\xi(g_\rho \rho(u_1)g_\rho^{-1})= \xi(\sigma(u_1))$ by (\ref{eqn:Qproj}), which implies that
\begin{equation}
 g_\rho\sigma_\beta(u_1) \sigma_\alpha(u_1)  \sigma_{\alpha+ \beta}(u_1) \sigma_{2\alpha+ \beta}(u_1)g_\rho^{-1}= \sigma_\beta(u_1) \sigma_\alpha(u_1)  \sigma_{\alpha+ \beta}(u_1) u_{2\alpha+ \beta}
\end{equation}
for some $u_{2\alpha+ \beta}\in U_{2\alpha+ \beta}$.  Another calculation using the commutation relations now shows that $z\sigma_\alpha(u_1)z^{-1}= \sigma_\alpha(u_1)$.  This implies that $z$ centralises both $U_\beta$ and $U_\alpha$, so $z\in Z(G)$.  By Lemma~\ref{lem:regular}(e), there exists $u_3\in U$ such that $\sigma_\beta(u_3)= 1$ and $\sigma_\alpha(u_3)\neq 1$.  Now $\xi(g_\rho \rho(u_3)g_\rho^{-1})= \xi(\sigma(u_3))$ by (\ref{eqn:Qproj}), so the commutation relations imply that $\xi(v_\beta\sigma(u_3) v_\beta^{-1})= \xi(\sigma(u_3))$, which is impossible by (\ref{eqn:noncomm}) as $v_\beta$ and $\sigma_\alpha(u_3)$ are nontrivial but $\sigma_\beta(u_3)$ is trivial.  This proves the claim.
 
 It follows from (\ref{eqn:Ru_conj}) that the representations $\tau|_U$ are pairwise $Z(L)R_u(P)$-conjugate to each other as $\tau$ ranges over the elements of $C_1$.  Lemma~\ref{lem:abelianpar} implies that the representations in $C_1$ are pairwise $Z(L)R_u(P)$-conjugate to each other.  This completes the proof.
\end{proof}

\begin{rem}
 We did not directly invoke Theorem~\ref{main_thm} in the above proof, but our argument amounts to checking that the fibres of $\widetilde{H}^1(\zeta)$ are finite.  Indeed, we prove that the representations $\tau|_U$ for $\tau\in C_1$, which {\em a priori} are only $P$-conjugate to each other, are in fact $Z(L)R_u(P)$-conjugate to each other.
\end{rem}

\begin{prop}
\label{prop:B2_localtoglobal}
 Let $\sigma\in {\rm Hom}(U,P)$ and let $C\subseteq \{\rho\in {\rm Hom}(H,P)\mid \rho|_U\ \mbox{is $G$-conjugate to $\sigma$}\}$.  Then $C$ is contained in a finite union of $G$-conjugacy classes.
\end{prop}

\begin{proof}
 By an argument like the one at the start of the proof of Proposition~\ref{prop:B2}, we can assume that $\sigma\in {\rm Hom}(U,P)'$ and $C\subseteq {\rm Hom}(H,P)'$.  Since ${\rm Hom}(H,L)_{\rm ir}$ is a finite union of $L$-conjugacy classes (Theorem~\ref{thm:cr_finite}), we can assume that $C\subseteq {\rm Hom}(H,P)'_\omega$ for some $\omega\in {\rm Hom}(H,L)_{\rm ir}$.  We separate the proof into cases.\smallskip\\
 (a) {\bf $\sigma$ is regular:}  By Lemma~\ref{lem:regular}(b), there exists $u_1\in U$ such that $\sigma(u_1)$ is regular.  Let $\rho\in C$.  Then $\rho|_U= g\cdot \sigma$ for some $g\in G$, so $\rho(u_1)= g\sigma(u_1)g^{-1}$ is regular.  Now $\sigma(u_1)$ and $\rho(u_1)$ both belong to $B$, by construction.  But a unipotent regular element of $G$ belongs to exactly one Borel subgroup of $G$ \cite[Ch.\ 4]{hum_conj}, so $gBg^{-1}= B$, so $g\in P$.  It follows from Proposition~\ref{prop:B2} that $C$ is contained in a finite union of $P$-conjugacy classes, as required.\smallskip\\
 (b) {\bf $\sigma$ is singular:}  Given $N\leq H$, set
 $$ C_N= \{\rho\in C\mid \rho^{-1}(R_u(P))= N\}. $$
 Note that if $C_N$ is nonempty---say, $\rho\in C_N$---then $N\unlhd H$ and $N$ is unipotent, as $\rho$ is faithful and $\rho(N)$ is unipotent, so $N\leq U$ (Corollary~\ref{cor:normaluniptmax}) and $N= (\rho|_U)^{-1}(R_u(P))$.  Since there are only finitely many possibilities for $\sigma^{-1}(R_u(P))$, it is enough to prove that $C_N$ is contained in a finite union of $G$-conjugacy classes, where $N:= \sigma^{-1}(R_u(P))$.  If $N= 1$ then $\rho^L= \omega$ belongs to ${\rm Hom}(H,L)'$.  But then ${\rm Hom}(H,P)_\omega$ is contained in a finite union of $P$-conjugacy classes by Lemma~\ref{lem:faithful}, so in this case we are done.  We can assume, therefore, that $N\neq 1$.  Clearly we can assume that $\sigma= \rho_0|_U$ for some $\rho_0\in C$.
 
 Let $\rho\in C$.  Choose $1\neq u_1\in N$.  Since $\rho$ is singular, $\rho_{\alpha+ \beta}(u_1)= 1$ by Lemma~\ref{lem:regular}(d), so $1\neq \rho(u_1)\in U_{2\alpha+ \beta}$.  Likewise, $\rho_0$ is singular, so $1\neq \sigma(u_1)= \rho_0(u_1)\in U_{2\alpha+ \beta}$.  After conjugating by an element of $Z(L)$ if necessary, we can assume that $\sigma(u_1)= \rho(u_1)$.  Let $A= \langle \sigma(u_1)\rangle$.  Choose $g\in G$ such that $(g\cdot \rho)|_U= \sigma$.  Then $g$ centralises $\sigma(u_1)$, so $g\in C_G(A)$.  Now $C_G(A)$ is non-reductive (see, e.g., \cite[Sec.\ 30.3]{hum}) and $C_G(A)$ is not contained in a Borel subgroup of $G$ because $C_G(A)$ contains $[L, L]$.  It follows from Lemma~\ref{lem:allconj} that $C_G(A)$ is contained in a unique proper parabolic subgroup $P'$ of $G$.  By a similar argument, $P'$ is the unique proper parabolic subgroup of $G$ that contains $C_G(U_{2\alpha+ \beta})$, and it is not hard to see that $P'= P$.  Hence $g\in P$.  It follows from Proposition~\ref{prop:B2} that $C$ is contained in a finite union of $P$-conjugacy classes.\smallskip\\
 This completes the proof.
\end{proof}

\begin{prop}
\label{prop:B2_p=2}
 Let $p= 2$, let $G$ be a simple group of type $B_2$ and let $H$ be finite.  Then $(G,H)$ is a K\"ulshammer pair.
\end{prop}

\begin{proof}
 As $p= 2$, the simply connected and adjoint forms of $B_2$ are isomorphic as abstract groups, so we can assume by Remark~\ref{rem:abstractiso} that $G$ is simply connected since $H$ is finite.  Fix $\sigma\in {\rm Hom}(U,G)$ and let $C\subseteq \{\rho\in {\rm Hom}(H,G)\mid \rho|_U\ \mbox{is $G$-conjugate to $\sigma$}\}$; we prove that $C$ is contained in a finite union of $G$-conjugacy classes.  By Theorem~\ref{thm:cr_finite}, we can assume that every $\rho\in C$ has image contained in some maximal parabolic subgroup of $G$, so we can assume that $C$ is contained in ${\rm Hom}(H,P_\alpha)\cup {\rm Hom}(H,P_\beta)$.  There is a bijective isogeny $f\colon G\ra G$ such that $f$ swaps long roots with short roots, and we can pick $f$ in such a way that $f(P_\beta)= P_\alpha$.  Hence we can assume that $C$ is contained in ${\rm Hom}(H,P_\beta)$.  The result now follows from Proposition~\ref{prop:B2_localtoglobal}.
\end{proof}

\begin{proof}[Proof of Theorem~\ref{thm:lowrank}]
 If $p= 0$ then the result follows from Theorem~\ref{thm:char0KQ}, so we assume that $p> 0$.  By Proposition~\ref{prop:subgrpcrit} and Remark~\ref{rem:reduction}, we can assume that $H$ is finite.  If $G$ is of type $B_2$ and $p= 2$ then the result follows from Proposition~\ref{prop:B2_p=2}, while if $G$ is of type $A_1$, $A_1\times A_1$ or $A_2$ then the result also follows (see Section~\ref{sec:intro}).  The only other possibilities are that $G$ is of type $B_2$ and $p\neq 2$, or $G$ is of type $G_2$ and $p\neq 2, 3$.  But then $p$ is good for $G$, so the result follows from \cite[I.5, Thm.\ 3]{slodowy}.
\end{proof}

\begin{rem}
 We do not know of any $H$ such that $H^0/R_u(H)$ is semisimple and $(G,H)$ is not a K\"ulshammer pair, where $G$ is of type $G_2$ and $p= 3$.
\end{rem}



\begin{thebibliography}{00}

\bibitem{BMR}
M.~Bate, B.~Martin, G.~R\"ohrle,
\emph{A geometric approach to complete reducibility},
Invent.\ Math.\ \textbf{161} (2005), no.\ 1, 177--218.

\bibitem{BMR_kuls}
M.~Bate, B.~Martin, G.~R\"ohrle,
\emph{On a question of {K}\"ulshammer for representations of finite groups in reductive groups},
Israel J. Math.\ \textbf{214} (2016), 463--470.

\bibitem{BMRT}
M.~Bate, B.~Martin, G.~R\"ohrle, R.~Tange,
\emph{Complete reducibility and separability},
Trans.\ Amer.\ Math.\ Soc.\ \textbf{362} (2010), no.\ 8, 4283--4311.

\bibitem{brown}
 K.S.~Brown,
 \emph{Cohomology of groups},
 Graduate Texts in Mathematics, vol.\ 87, Springer, 1982, x+306 pp.

\bibitem{cram}
G.-M.~Cram,
\emph{On a question of K\"ulshammer about algebraic group actions: an example}, appendix to \cite{slodowy},
Austral.\ Math.\ Soc.\ Lect.\ Ser.\ \textbf{9},
Algebraic groups and Lie groups, 346--348,
Cambridge Univ.\ Press, Cambridge, 1997.

\bibitem{guralnick}
R.M.~Guralnick,
\emph{Intersections of conjugacy classes and subgroups of algebraic groups},
Proc.\ Amer.\ Math.\ Soc.\ \textbf{135} (2007), no.\ 3, 689--693.

\bibitem{hum}
J.E.~Humphreys,
\emph{Linear algebraic groups},
Springer-Verlag, New York, 1975.

\bibitem{hum_conj}
J.E.~Humphreys,
\emph{Conjugacy classes in semisimple algebraic groups},
Mathematical Surveys and Monographs \textbf{43}, American Mathematical Society, Providence, RI, 1995.

\bibitem{kempf}
G.R.~Kempf,
\emph{Instability in invariant theory},
Ann.\ Math.\ \textbf{108} (1978), no.\ 2, 299--316.

\bibitem{king}
O.H.~King,
\emph{The subgroup structure of finite classical groups in terms of geometric configurations},
In: Surveys in combinatorics 2005, 29--56, 
London Math.\ Soc.\ Lecture Note Ser.\ {\bf 327}, Cambridge Univ.\ Press, Cambridge, 2005. 

\bibitem{kuls}
 B.~K\"ulshammer,
 \emph{{D}onovan's conjecture, crossed products and algebraic group actions},
 Israel J.\ Math.\ \textbf{92} (1995), no.\ 1--3, 295--306.

\bibitem{liebeckseitz0}
M.W.~Liebeck, G.M.~Seitz,
\emph{Reductive subgroups of exceptional algebraic groups},
Mem.\ Amer.\ Math.\ Soc.\ no.\ \textbf{580} (1996).

\bibitem{liebeckseitz1}
M.W.~Liebeck, G.M.~Seitz,
\emph{Unipotent and nilpotent classes in simple algebraic groups and Lie algebras},
Mathematical Surveys and Monographs \textbf{180}, Amer.\ Math.\ Society, Providence, RI, 2012.

\bibitem{liebecktesterman}
M.W.~Liebeck, D.M.~Testerman,
\emph{Irreducible subgroups of algebraic groups},
Q. J. Math.\ \textbf{55} (2004), 47--55.

\bibitem{lond}
D.~Lond,
\emph{On reductive subgroups of algebraic groups and a question of {K}\"ulshammer},
PhD thesis, University of Canterbury, 2013.

\bibitem{martin}
B.M.S.~Martin,
\emph{Reductive subgroups of reductive groups in nonzero characteristic},
J. Algebra \textbf{262} (2003), no.\ 2, 265--286.

\bibitem{mostow}
G.D.~Mostow,
\emph{Fully reducible subgroups of algebraic groups},
American J. Math.\ \textbf{78} (1956), no.\ 1, 200--221.

\bibitem{nagata}
M.~Nagata,
\emph{Complete reducibility of rational representations of a matric group},
J. Math.\ Kyoto University \textbf{1} (1961), 87--99.

\bibitem{newstead}
P. E.~Newstead,
\emph{Introduction to moduli problems and orbit spaces},
Tata Institute of Fundamental Research Lectures on Mathematics and Physics \textbf{51}.
Tata Institute of Fundamental Research, Bombay, 1978.

\bibitem{rich}
R.W.~Richardson,
\emph{Affine coset spaces of reductive algebraic groups},
Bull.\ London Math.\ Soc.\ \textbf{9} (1977), no.\ 1, 38--41.

\bibitem{rich2}
R.W.~Richardson,
\emph{Conjugacy classes in {L}ie algebras and algebraic groups},
Ann.\  Math.\  \textbf{86} (1967), 1--15.

\bibitem{rich3}
R.W.~Richardson,
\emph{On orbits of algebraic groups and {L}ie groups},
Bull.\ Aust.\ Math.\ Soc.\ \textbf{25} (1982), 1--28.

\bibitem{serre1}
J.-P. Serre,
\emph{The notion of complete reducibility in group theory},
Moursund Lectures, Part II, University of Oregon, 1998,\
{\tt arXiv:math/0305257v1 [math.GR]}.

\bibitem{serre2}
J-P.~Serre,
\emph{Compl\`ete r\'eductibilit\'e},
S\'eminaire Bourbaki, 56\`eme ann\'ee, 2003--2004, n$^{\rm o}$ 932, 2005.

\bibitem{slodowy}
P.~Slodowy,
\emph{Two notes on a finiteness problem in the representation theory
of finite groups},
Austral.\ Math.\ Soc.\ Lect.\ Ser.\ \textbf{9},
Algebraic groups and Lie groups, 331--348,
Cambridge Univ.\ Press, Cambridge, 1997.

\bibitem{stewart_PhD}
D.I.~Stewart,
\emph{$G$-complete reducibility and the exceptional algebraic groups},
PhD thesis, Imperial College London, 2010.

\bibitem{stewartG2}
 D.I.~Stewart,
 \emph{The reductive subgroups of $G_2$},
 J. Group Theory \textbf{13} (2010), no.\ 1, 117--130.
   
\bibitem{stewartF4}
   D.I.~Stewart,
   \emph{The reductive subgroups of $F_4$},
   Mem.\ Amer.\ Math.\ Soc.\ \textbf{223} (2013), no.\ 1049, vi+88pp.
   
\bibitem{stewartTAMS}
 D.I.~Stewart,
 \emph{On unipotent algebraic $G$-groups and 1-cohomology},
 Trans.\ Amer.\ Math.\ Soc.\ \textbf{365} (2013), no.\ 12, 6343--6365.
  
\bibitem{uchiyama2}
T.~Uchiyama,
\emph{Separability and complete reducibility of subgroups of the Weyl group of a simple algebraic group of type $E_7$}, J.\ Algebra \textbf{422} (2015), 357--372. 

\bibitem{uchiyama3}
T.~Uchiyama,
\emph{Non-separability and complete reducibility: $E_n$ examples with an application to a question of K\"ulshammer}, to appear in J. Group Theory.
    
\end{thebibliography}

\end{document}